\documentclass[12pt]{article} 
\usepackage[utf8x]{inputenc}
\usepackage{tikz} 
\usepackage[all]{xy}
\usepackage{authblk}
\usepackage[babel=true]{csquotes}
\usepackage{amsmath,amsthm, amssymb,amsfonts}
\usepackage[hmargin=1in,vmargin=1in]{geometry}
\numberwithin{equation}{subsection} 
\xyoption{arc}
\usepackage{eucal}
\usepackage{indentfirst}
\usepackage{mathrsfs}
\usepackage{verbatim}
\usepackage{makeidx}
\usepackage{graphicx}
\usepackage{hyperref}
\usepackage{enumitem} 
\usepackage{extpfeil} 
\usepackage{dsfont}
\usepackage[T1]{fontenc}
\newtheorem{thm}{Theorem}[section]
\newtheorem{prop}[thm]{Proposition}
\newtheorem{pdef}[thm]{Proposition-Definition}
\newtheorem{lem}[thm]{Lemma}
\newtheorem*{thmsans}{Theorem}
\newtheorem*{princ}{Principium}

\newtheorem{cor}[thm]{Corollary}

\newtheorem{conj}[thm]{Conjecture}

\newtheorem*{lemsans}{Lemma}
\newtheoremstyle{bidule}
{3pt}
{3pt}
{}
{}
{\scshape}
{.}
{.5em}
{}
\newtheorem{df}[thm]{Definition}
\theoremstyle{definition}
\newtheorem{rmk}[thm]{Remark}

\newtheorem{schol}[thm]{Scholium}
\newtheorem*{defsans}{Definition} 
 
\newtheorem*{term}{Terminology}
\newtheorem*{note}{Note}

\newtheorem*{warn}{Warning}
\newtheorem*{claim}{Claim}
\newtheorem{nota}[thm]{Notation}
\newtheorem*{conv}{Convention}

\newcommand{\Ev}{\tx{Ev}}
\newcommand{\C}{\mathcal{C}}
\newcommand{\Ub}{\mathcal{U}}

\newcommand{\F}{\mathcal{F}}

\newcommand{\Ar}{\text{Arr}}

\newcommand{\D}{\mathcal{D}}
\newcommand{\Ba}{\mathcal{B}}

\newcommand{\Aa}{\mathcal{A}}
\newcommand{\B}{\mathscr{B}}

\newcommand{\M}{\mathscr{M}}
\newcommand{\V}{\mathbb{V}}
\newcommand{\Ja}{\mathbf{J}} 
\newcommand{\J}{\mathcal{J}} 

\newcommand{\T}{\mathcal{T}}

\newcommand{\Ea}{\mathcal{E}} 
 
\newcommand{\W}{\mathscr{W}}

\newcommand{\I}{\mathbf{I}}
\newcommand{\Un}{\mathbb{I}} 

\newcommand{\G}{\mathcal{G}}

\renewcommand{\H}{\mathcal{H}}

\renewcommand{\S}{\mathcal{S}}

\newcommand{\Fb}{\mathbf{F}}
\newcommand{\Sim}{\mathscr{S}}

\newcommand{\Pa}{\mathcal{P}}

\renewcommand{\to}{\longrightarrow}
\newcommand{\ol}{\overline}
\newcommand{\ul}{\underline}

\renewcommand{\bf}{\mathbf}
\newcommand{\U}{\mathbb{U}}

\newcommand{\n}{\textbf{n}} 
\newcommand{\m}{\textbf{m}} 
\newcommand{\p}{\textbf{p}} 
\newcommand{\q}{\textbf{q}} 
\renewcommand{\k}{\textbf{k}} 
\newcommand{\0}{\textbf{0}} 
\renewcommand{\1}{\textbf{1}} 
\renewcommand{\2}{\textbf{2}} 
\newcommand{\uc}{\mathds{1}} 

\newcommand{\tx}{\text}
\newcommand{\tld}{\widetilde}

\renewcommand{\to}{\longrightarrow}
\DeclareMathOperator\Id{Id}
\DeclareMathOperator\Hom{Hom}

\DeclareMathOperator\Lax{Lax}
\DeclareMathOperator\Cat{\mathbf{Cat}}
\DeclareMathOperator\msx{\M_{\S}(\tx{$X$})}
\DeclareMathOperator\msxsu{\M_{\S}(\tx{$X$})_{\tx{$u$}}}

\DeclareMathOperator\msxeasy{\msx_{\tx{$easy$}}}

\DeclareMathOperator\kb{\mathbf{K}} 

\DeclareMathOperator\Lb{\tx{L}} 

\DeclareMathOperator\sx{\S_{\ol{X}}} 
\DeclareMathOperator\Ho{\mathbf{ho}} 

\DeclareMathOperator\Depipl{\Delta_{\tx{epi}}^+}

\DeclareMathOperator\laxlatch{\mathbf{Latch}_{lax}}  
 
\DeclareMathOperator\av{\alpha_{\downarrow_{\Id_V}}}
 
\DeclareMathOperator\Phepi{\Phi_{\tx{epi}}} 
\DeclareMathOperator\Phepipl{\Phi_{\tx{epi}+}} 
\DeclareMathOperator\Phepiop{\Phi_{\tx{epi}}^{\tx{op}}}
\DeclareMathOperator\Phepiopl{\Phi_{\tx{epi}+}^{\tx{op}}}

\DeclareMathOperator\ap{\tx{-}} 
\DeclareMathOperator\bp{\tx{$+$}} 
 
\DeclareMathOperator\fset{\textbf{FinSet}} 
 
\DeclareMathOperator\coms{\tx{Com}_{\S}(\M)} 
\DeclareMathOperator\comsec{\tx{Com}_{\S}(\M)_{\tx{$e$}}^{\mathbf{c}}}
\DeclareMathOperator\comse{\tx{Com}_{\S}(\M)_{\tx{$e$}}}  
\DeclareMathOperator\comsep{\tx{Com}_{\S}(\M)_{\tx{$e+$}}}
\DeclareMathOperator\comsepc{\tx{Com}_{\S}(\M)_{\tx{$e+$}}^{\mathbf{c}}}  
\DeclareMathOperator\com{\tx{Com}(\M)} 
 
\DeclareMathOperator\ox{\otimes}
\DeclareMathOperator\Add{\tx{$Add$}}

\title{Model structure on commutative\\
co-Segal dg-algebras
 in characteristic $p>0$} 
\author{Hugo V. Bacard 
}
 \affil{Western University}
\date{\today}
\begin{document}
\maketitle
\begin{abstract}
We study weak commutative algebras in a symmetric monoidal model category $\M$. We provide a model structure on these algebras for any symmetric monoidal model category that is combinatorial and left proper.  Our motivation was to have a homotopy theory  of  weak  commutative dg-algebras in characteristic $p>0$, since there is no such theory for strict commutative dg-algebras. For a general $\M$, we show that if the projective model structure on strict commutative algebras exists, then the inclusion from strict to weak algebras is a Quillen equivalence. The results of this paper can be generalized to symmetric co-Segal $\Pa$-algebras for any operad $\Pa$. And surprisingly, the axioms of a monoidal model category are not necessary to get the model structure on co-Segal commutative algebras.
\end{abstract}
\setcounter{tocdepth}{1}
\tableofcontents 
\newpage
\section{Introduction}
Let $\M=(\ul{M}, \otimes, I)$ be a symmetric monoidal model category. In this paper we consider a notion of weak commutative monoid (algebra) in $\M$ that we will call \emph{commutative co-Segal algebra}. Our main motivation is when $\M$ is the category $C(\k)$  of unbounded chain complexes over a commutative ring $\k$, regardless of the characteristic of $\k$.\\

Recall that  $C(\k)$  has the structure of a monoidal model category in which the fibrations are the degree-wise surjective maps and the weak equivalences are the quasi-isomorphisms. The reader will find a good exposition of the model structure on $C(\k)$ in Hovey's book \cite{Hov-model}. If $\k$ is a field and if  $\Pa$ is a (symmetric) operad enriched over $C(\k)$, Hinich \cite{Hinich_haha} showed that there is a model structure on $\Pa$-algebras if $\Pa$ is \emph{ $\Sigma$-split}, where $\Sigma$ is the (generic) symmetric group. And he shows that if $\k$ is a field of characteristic $0$ then any $\Pa$ is automatically $\Sigma$-split.\\

The homotopy theory of commutative monoids in a general symmetric monoidal model category $\M$ was extensively studied by White \cite{White_com}. White gave sufficient conditions on $\M$ under which the \emph{natural} model structure on commutative monoids in $\M$ exists (\cite[Theorem 3.2]{White_com}).\\

Basically the major problem is the action of the symmetric group and its interaction with the tensor product $\otimes$. The issue is that taking the quotient object is not a functor that usually preserves weak equivalences of underlying objects. In general, it doesn't even preserve trivial cofibrations. These quotients appear in the free-algebras functor, along which we usually transfer the homotopy theory from $\M$ to the category of commutative algebras.\\

The guiding principle of this paper is the following fact, that we will take as a `motto'.
\begin{princ} The co-Segal formalism allows us to put the action of the symmetric groups on the objects of the homotopy category $\Ho(\M)$, and not on the objects of $\M$ themselves.
For example if $\Pa$ is a $C(\k)$-operad, a co-Segal $\Pa$-algebra structure on $A \in \M$ will be given by a zigzag in $\M$:
$$\Pa(l) \otimes_{\Sigma_l} A^{\otimes l} \to A(l) \xleftarrow{\sim} A.$$

\begin{itemize}[label=$-$]
\item By doing this we avoid a direct interaction in $\M$ between the symmetric group and the object $A$ that will carry the weak algebra structure.
\item The definition of a co-Segal $\Pa$-algebra, does not require $A$  to interact with the product $\otimes$, in the sense that, a priori, there is no map  $\{\cdots \} \to A $  in $\M$, such that the source $\{\cdots\}$ is  a nontrivial expression involving $\otimes$.
\item Consequently we don't use the axioms of a monoidal model category $\M$, to have a Quillen model structure for co-Segal $\Pa$-algebras.
\item We only need a symmetric monoidal closed category $\M=(\ul{M}, \otimes,I)$ whose underlying category $\ul{M}$ is a model category that is combinatorial, and most importantly left proper.
\end{itemize}
\end{princ}
We explain briefly in the next part of this introduction how this works. Although everything works for any operad $\Pa$, we only treat in this paper the case corresponding to $\Pa=Com$, the operad of commutative monoids. This is only to avoid a very long paper.

 We also do this implicitly without really using the language of operads. We use instead symmetric monoidal lax functor. The general case will be done properly later. But for now, we refer the reader to the author's preprint \cite{Bacard_lpcosegx} about how we generalize this to co-Segal $\Pa$-algebras for an operad $\Pa$.

\subsection{How the theory works}
Let $\fset$ be the category of finite sets, and let $\Phi$ be the equivalent model for $\fset$ formed by the sets $\n=\{0,...,n-1 \}$ and all functions between them. Let $\Phepipl \subset \Phi$ be the subcategory of epimorphisms in which we've excluded the empty set; and finally consider the opposite category $\Phepiopl$. One can already observe that Segal's $\Gamma$-category is implicitly involved: this is the underlying idea.\\ 

Now if $A$ is an object of $\M$, we usually define an algebra structure on $A$ by specifying a multiplication $\mu:A \ox A \to A $ and a unit $e: I \to A$. And we demand the usual axiom of associativity and commutativity for $\mu$ and the unitality axiom for $e$ and $\mu$.
 
If we think in the co-Segal world, we will regard this as a co-Segal algebra structure on a object $A$ that is  \emph{static} i.e, it doesn't ``change with time''. 

More precisely, having a co-Segal algebra structure on $A$ is the possibility to have a multiplication $A \ox A \to A(2)$ that does \ul{not} necessarily land to $A$ itself but to another object $A(2)$ that possesses the same homotopy information as $A$. \\

The object $A$ is the initial entry (``initial state'') of a diagram of weak equivalences, defined over $\Phepiopl$:  
 $$A=A(1) \xrightarrow{\sim} A(2) \rightrightarrows \cdots A(n) \cdots.$$

The category $\Phepiopl$ is a direct category, and being direct  is a concept that reflects a one-way \emph{evolution}. Therefore we can think of this diagram as an evolution of $A$ without the possibility to go back in time (...can we ?). And since this is a diagram of weak equivalences, these changes preserve the entire homotopy information of the initial entry \footnote{Like in real life where a growing subject won't change its DNA (fortunately...)}  $A$. In mathematical terms this diagram is a Reedy resolution of $A$ (see \cite{Hirsch-model-loc}). And the idea of a co-Segal structure is to have an algebra structure `as you go' on the resolution of $A$ rather than on $A$ itself.\\

It turns out that having a pseudo-multiplication  $A \ox A \to A(2)$ plus a direct diagram of weak equivalences (co-Segal conditions) give more flexibility for homotopy theory purposes. Indeed all we have to do is to \emph{concentrate} first the homotopy theory  at the initial entries, and then use left Bousfield localizations. We will be specific below about what  `concentrate the homotopy theory' means.\\

The model structure on co-Segal commutative monoids is close to be a  `\emph{Dugger replacement}' (see \cite{Dugger_rep}), for the model structure on usual strict commutative monoids (when it exists). The comparison between strict commutative monoids and co-Segal commutative monoids  can be seen as an algebraic version of the inclusion  $\M \hookrightarrow \Hom(\Phepiopl, \M)$ that identifies $\M$ with the full subcategory of constant diagrams.\\

The objects of $\Hom(\Phepiopl, \M)$ that correspond to diagrams of weak equivalences are by definition constant in the homotopy category $\Ho(\M)$ (the co-Segal diagrams). And it can be shown either with classical methods or using the techniques of this paper, that if $\M$ is combinatorial and left proper then this inclusion induces an equivalence of homotopy categories between $\M$ and the subcategory of diagrams that satisfies the co-Segal conditions (diagrams of weak equivalences).\\

To close this introduction, we list hereafter the key points that lead to the homotopy theory of co-Segal structures. It's the same process for both symmetric and nonsymmetric structures. 
\begin{enumerate}
\item First we define a co-Segal commutative monoid as a symmetric monoidal lax diagram $$\F: (\Phepiop,+,\0) \to \M,$$ 
that satisfies the \emph{co-Segal conditions}. If the co-Segal conditions are dropped we will say that $\F$ is a co-Segal commutative \emph{premonoid}.
\item There is a category $\coms$ of all commutative premonoids, that is locally presentable if $\M$ is. In particular $\coms$ is complete and cocomplete.
\item We have an adjunction given by the evaluation at $\1$:
$$\Ev_\1: \coms \leftrightarrows \M: Free.$$
\item We get a right-induced `easy' model structure on $\coms$  because:
\begin{itemize}[label=$-$]
\item The right adjoint $\Ev_\1$ happens to preserve pushouts
\item And more importantly, the composite $\Ev_\1 \circ Free: \M \to \M$ preserves cofibrations and trivial cofibrations. This is one of the major differences between usual commutative monoids and the co-Segal ones.
\item In this model structure a map $\sigma: \F \to \G$ is a weak equivalence (resp. a fibration) if the component $\sigma_\1: \F(\1)\to \G(\1)$ is a weak equivalence (resp. a fibration). This is what we mean by ``concentrate the homotopy theory'' at the initial entry $\1$.
\end{itemize}
\item This easy model category $\comse$ is left proper if $\M$ is. And thanks to Smith's theorem for left proper combinatorial model category, we can form left Bousfield localizations (later).
\item If $\com$ is the category of usual strict commutative monoids, then there is an inclusion $\iota: \com \hookrightarrow \coms$ that exhibits $\com$ as a full reflective subcategory. In particular there is a left adjoint $|-|: \coms \to \com$.
\item If the natural model structure on $\com$ exists as in \cite{White_com}, then it agrees with the easy model structure $\comse$. Putting this differently, we have a Quillen adjunction in which $\iota: \com \hookrightarrow \comse$ is a right Quillen functor.
\item If $\I$ is a set of generating cofibrations for $\M$ then there exists a set of maps $\kb(\I)$ in $\coms$ such that:
\begin{itemize}[label=$-$]
\item Every $\kb(\I)$-injective commutative premonoid $\F$ satisfies the co-Segal conditions.
\item The left Quillen functor $|-|: \coms \to \com$ sends elements of $\kb(\I)$ to isomorphisms (therefore to weak equivalences).
\end{itemize}
\item After this we build another model structure `$+$' on $\coms$ with the same weak equivalences but such that every element of $\kb(\I)$ becomes a generating cofibration. 
\item The identity functor gives a left Quillen functor $\comse \to \comsep$ that is a Quillen equivalence.
\item There is a bit technical part which says for any $\F \in \coms$ the unit map $\F \to \iota(|\F|)$ is a $\kb(\I)$-local equivalence in $\comse$ (whence in $\comsep$).
\item Finally we introduce the left Bousfield localizations of $\comse$ and $\comsep$ with respect to the same set $\kb(\I)$.
\end{enumerate}
After these steps we get our main result as follows (Theorem \ref{main-thm-quillen-equiv}).
\begin{thmsans}\label{main-thm-quillen-equiv-intro}
Let $\M$ be a symmetric monoidal model category that is combinatorial and left proper. Then the following hold. 
\begin{enumerate}
\item The left Quillen equivalence $\comse \to \comsep$ induced by the identity of $\coms$ descends to a left Quillen equivalence between the respective left Bousfield localizations with respect to $\kb(\I)$:
$$\comsec \to \comsepc.$$
\item In the model category $\comsepc$ any fibrant object is a co-Segal commutative monoid.
\item If the transferred model structure on $\com$ exists, then the adjunction
$$ |-|^{\mathbf{c}}: \comsec \leftrightarrows \com: \iota,$$
is a Quillen equivalence. 
\item The diagram $\comsepc \xleftarrow{}\comsec \xrightarrow{|-|^{\mathbf{c}}} \com$ is a zigzag of Quillen equivalences. In particular we have a diagram of equivalences between the homotopy categories.
$$\Ho[\comsepc] \xleftarrow{\simeq} \Ho[\comsec] \xrightarrow{\simeq} \Ho[\com].$$
\end{enumerate}
\end{thmsans}

\begin{note}
It can be shown using the HELP Lemma (Homotopy Extension Lifting Property), that if $\M$ is tractable (see \cite{Barwick_localization,Simpson_HTHC}), then any fibrant object in the model category $\comsec$ is a co-Segal commutative monoid. But we didn't treat this here because our motivation was to keep the hypotheses on $\M$ as minimal as possible.
\end{note}

The material provided here should work \ul{as is}, for any operad $\Pa$. But since we have not written this yet, we only include a conjecture. The reader will find a definition of nonsymmetric co-Segal $\Pa$-algebra in \cite{Bacard_lpcosegx}.

\begin{conj}
For any operad $\Pa$, the results of Theorem \ref{main-thm-1}, Theorem \ref{main-thm-2} and Theorem \ref{main-thm-quillen-equiv} hold for co-Segal commutative $\Pa$-algebras.
\end{conj}

\subsection{Organization of the paper}

The material of Sections \ref{sec-def} and \ref{sec-prop} provides the definition of the objects we're studying, as well as the categorical properties of the category $\coms$ of all co-Segal commutative premonoids.

In Section \ref{sec-transf}, we establish the first model structure on $\coms$ (called easy model structure) . We also outline how it agrees with the model structure on usual commutative monoids.\\

Then we introduce in Section \ref{section-loc-set} the set $\kb(\I)$ that will be used to localize the previous model structure.\\

The material of Section \ref{section-2-constant} is very important, and unfortunately a bit technical. The main result there is the content of Proposition \ref{pdef-deux-constant}.\\

After this, we put a new model structure on $\coms$ in Section \ref{section-new-model-str}. And we discuss the left Bousfield localizations of the two model categories in Section \ref{sec-bousfield}.\\
Finally we compare the various homotopy theories in Section \ref{sec-Quillen-equiv} with our main theorem (Theorem \ref{main-thm-quillen-equiv}).

\begin{note}
We would like to warn the reader that most \emph{forgetful functors} will be denoted by the same letter  `$\Ub$'.
\end{note}
\newpage
\section{Commutative co-Segal algebras}\label{sec-def}
Let $\M=(\ul{M},\otimes, I)$ be a symmetric monoidal closed category, regarded as a $2$-category with a single object. The example we have in mind for this paper is precisely $\M= (C(\k), \otimes_\k,\k)$.

\begin{nota}
\begin{enumerate}
\item Let $\n=\{0,...,n-1\}$ be the usual $n$-elements set. In particular $\0$ is the empty set.
\item Let $(\Phi,+,\0)$ be the symmetric monoidal category formed by the sets $\n=\{0,...,n-1\}$ with all functions between them. The monoidal structure $+$ is the disjoint union.  
\item Let $(\Phi_{\tx{epi}},+,\0)$ be the symmetric monoidal subcategory of  $(\Phi,+,\0)$ having the same objects but only surjective functions. Note that $\0$ is \emph{isolated} there i.e, there is no epimorphism involving $\0$ other than the identity. The only reason we keep it is to avoid a language of semi-monoidal category $(\Phepi,+)$. But when it's convenient we will ignore the presence of $\0$ in $\Phepi$.
\item Let $\Phepipl \subset \Phepi$ be the full subcategory not containing $\0$.
\item Let $(\Delta^+, +,0)$ be the category of finite ordinals and order preserving maps.
\item  Let $(\Depipl,+, 0)$  be the subcategory of epimorphisms. 
\item  In each case we have a functor that forgets the order:
$$ (\Delta^+, +,0) \to (\Phi,+,\0),$$
$$(\Depipl,+, 0) \to (\Phi_{\tx{epi}},+,\0)$$
\item We will denote by $(\Phi^{\tx{op}},+,\0)$ and  $(\Phepiop,+,\0)$ the corresponding opposite symmetric monoidal categories.
\end{enumerate}
\end{nota}
\begin{note}
We assume that the reader is familiar with the notion of symmetric monoidal  lax functor. 
Given $\M=(\ul{M},\otimes, I)$ as above, we will identify $\M$ and the underlying category $\ul{M}$ when it's convenient.
\end{note}

\begin{df}
Let $\M$ be a symmetric monoidal (model) category. A co-Segal commutative monoid $\F$ of $\M$ is a symmetric monoidal lax functor 
$$\F: (\Phepiop,+,\0) \to \M, $$
that satisfies the following conditions. 
\begin{enumerate}

\item $\F$ is a normal lax functor:
\begin{itemize}[label=$-$]

\item The map $I \to \F(\0)$ is an isomorphism;
\item The following maps are natural isomorphisms.
$$\F(\0) \otimes \F(\n)\xrightarrow{\cong} \F(\n).$$
\end{itemize}
\item $\F$ satisfies the co-Segal conditions, i.e, the underlying functor 
$$\F: \Phepiopl \to \M $$
is a diagram of weak equivalences:
 $$\F(\1) \xrightarrow{\sim} \F(\2) \rightrightarrows \cdots \F(\n) \cdots.$$ 
\item $\F$ is weakly unital i.e, there exists a map $e:I \to \F(\1)$ such that the following commutes.
\begin{equation}\label{diag-unitality}
\xy
(0,18)*+{I \otimes \F(\1)}="W";
(0,0)*+{\F(\1) \otimes \F(\1) }="X";
(40,0)*+{\F(\2)}="Y";
(40,18)*+{\F(\1)}="E";
{\ar@{->}^-{\varphi}"X";"Y"};
{\ar@{->}^-{e \otimes \Id}"W";"X"};
{\ar@{->}^-{\cong}"W";"E"};
{\ar@{->}^-{\sim}"E";"Y"};
\endxy
\end{equation}
\end{enumerate}
If $\F$ doesn't satisfies the co-Segal conditions, then $\F$ will be called a commutative (co-Segal) premonoid.
\end{df}

\begin{conv}
\begin{enumerate}
\item We will always assume that $\F(\0)=I$ and that the map $I \to \F(\0)$ is the identity. 
\item Given a functor $\F: \Phepiopl \to \M$, we will implicitly extend it to a functor defined over the entire category $\Phepiop$ whose value at $\0$ is $I$. Similarly we will extend any natural transformation by letting the component at $\0$ be the identity $\Id_I: I \to I$.
\end{enumerate}

\end{conv}
We need to say what  the morphisms of co-Segal premonoids are. 
\begin{df}
Let $\F$ and $\G$ be commutative co-Segal premonoids. A morphism $\sigma: \F \to \G$ is a natural transformation $\sigma \in \Hom(\Phepiop,\M)$ that satisfies the following conditions.
\begin{enumerate}
\item For every $\n,\m$ the following commutes.
\begin{equation}
\xy
(0,18)*+{\F(\n) \otimes \F(\m)}="W";
(0,0)*+{\G(\n) \otimes \G(\m) }="X";
(40,0)*+{\G(\n+\m)}="Y";
(40,18)*+{\F(\n+\m)}="E";
{\ar@{->}^-{\varphi}"X";"Y"};
{\ar@{->}^-{\sigma \otimes \sigma}"W";"X"};
{\ar@{->}^-{\varphi}"W";"E"};
{\ar@{->}^-{\sigma}"E";"Y"};
\endxy
\end{equation}
\item The following commutes.
\begin{equation}
\xy
(10,18)*+{I}="W";
(0,0)*+{\F(\1)}="X";
(20,0)*+{\G(\1)}="Y";
{\ar@{->}^-{\sigma}"X";"Y"};
{\ar@{->}_-{e}"W";"X"};
{\ar@{->}^-{e}"W";"Y"};
\endxy
\end{equation}
\end{enumerate}
We will denote by $\coms$ the category of co-Segal premonoids and morphisms between them. There is a forgetful functor
$$\Ub: \coms \to \Hom(\Phepiopl,\M).$$
\end{df}

The following proposition will be given without a proof, because it's a straightforward application of the definition. 
\begin{prop} 
Let $\M$ be a symmetric monoidal (model) category.
\begin{enumerate}
\item We have an equivalence between the following data.
\begin{itemize}[label=$-$]
\item A commutative co-Segal monoid $\F$ such that the diagram $\F: \Phepiopl \to \M$ is a constant diagram.
\item A usual strict commutative monoid of $\M$.
\end{itemize}
\item Let $\com$ be the category of usual strict commutative monoids. Then we have a fully faithful inclusion functor
$$\iota: \com \to \coms.$$
\end{enumerate}
\end{prop}

\begin{warn}
We would like to warn the reader about our notation. As a general rule, whenever there is a subscript `$\S$', it means that we are working in co-Segal settings.
\end{warn}

\begin{rmk}
If we follow Hinich's notation `$\#$' for the forgetful functor (see \cite{Hinich_haha}), we have a commutative diagram:
\begin{equation}
\xy
(0,18)*+{\com}="W";
(0,0)*+{\coms }="X";
(40,0)*+{\Hom(\Phepiopl,\M)}="Y";
(40,18)*+{\M}="E";
{\ar@{->}^-{\#}"X";"Y"};
{\ar@{->}^-{\iota}"W";"X"};
{\ar@{->}^-{\#}"W";"E"};
{\ar@{->}^-{}"E";"Y"};
\endxy
\end{equation}

The functor $\M \hookrightarrow \Hom(\Phepiopl,\M)$ is the functor that takes an object $m \in \M$ to the constant diagram of value $m$. In our case since $\1$ is initial in $\Phepiopl$, this functor is the left adjoint  to the evaluation functor at $\1$. Following Hirschhorn's notation in \cite{Hirsch-model-loc}, we should denote by $\Fb^{\1}$ this functor. In particular we have an adjunction
$$\Ev_\1:\Hom(\Phepiopl,\M) \leftrightarrows \M: \Fb^{\1}.$$
\end{rmk}

\section{Categorical properties of commutative premonoids}\label{sec-prop}
The category $\Phepiopl$ is a direct category that has an initial object corresponding to $\1$. The advantage of such category is that we can modify diagram indexed by such category as outline hereafter.
\begin{schol} \label{schol1}
\
\begin{enumerate}
\item Let $\J$ be a \ul{direct} category with an initial object $j_0$ and let $F: \J \to \B$ be a functor. Then given any morphism $h:b \to F(j_0)$ in $\B$, we can define a diagram $h^\star F: \J \to \B$ by simply changing the value $F(j_0)$ to $b$.  More precisely $h^\star F$ is the composite of the chain of functors:
$$\J \xrightarrow{\simeq} (j_0\downarrow \J) \xrightarrow{F} (F(j_0)\downarrow \B)\xrightarrow{h^\star} (b\downarrow \B)\to \B.$$
We have a natural transformation $h^\star F \to F$ whose component at $j_0$ is $h$. Every other component $[h^\star F] (j) \to F(j)$ is the identity map.

\item Let $\F: (\Phepiopl,+,\0) \to \M$ be a commutative premonoid in $\M$ and let $h: m \to \F(\1)$ be any morphism in $\M$. Assume that:
\begin{itemize}[label=$-$]
\item there is a map  $\tld{e}: I \to m$, and
\item this map fits in a factorization of $e:I \to \F(\1)$ as:
$$  I \xrightarrow{e} \F(\1)= I \xrightarrow{\tld{e}} m \xrightarrow{h} \F(\1).$$
\end{itemize}
Then the functor $h^\star \F$ forms a symmetric monoidal normal lax functor $$h^\star \F: (\Phepiopl,+,\0)  \to \M,$$ 
and there is a transformation of lax functors $h^\star \F \to \F$ induced by the natural transformations $h^\star \F \to \F$.

Indeed, we get the laxity maps involving the new object $m$ as:
$$m \otimes \F(\n) \xrightarrow{h \otimes \Id} \F(\1) \otimes \F(\n) \to \F(\1 + \n);$$
$$m \otimes m \xrightarrow{h \otimes h} \F(\1) \otimes \F(\1) \to \F(\2);$$
$$ \F(\n) \otimes m \xrightarrow{\Id \otimes h}\F(\n) \ox \F(\1) \to \F(\n +\1).$$
It takes a little effort to see that these laxity maps remain associative and compatible with the symmetry of $\otimes$.
\item The lax diagram $h^{\star}\F$ satisfies the unitality condition as in the diagram \eqref{diag-unitality}. Indeed everything commutes in the diagram below.
\[
\xy
(0,15)*+{I \otimes m}="A";
(30,15)*+{m}="B";
(0,0)*+{I \otimes \F(\1)}="C";
(-30,0)*+{m \otimes m}="E";
(30,0)*+{\F(\1)}="D";
(0,-15)*+{\F(\1) \otimes \F(\1)}="X";
(30,-15)*+{\F(\2)}="Y";
{\ar@{->}^-{\cong}"A";"B"};
{\ar@{->}^-{\Id \otimes h}"A";"C"};
{\ar@{->}^-{h}"B";"D"};
{\ar@{->}^-{\cong}"C";"D"};
{\ar@{->}^-{}"X";"Y"};
{\ar@{->}^-{e \otimes \Id}"C";"X"};
{\ar@{->}^-{}"D";"Y"};
{\ar@{->}_-{\tld{e} \otimes \Id}"A";"E"};
{\ar@{->}_-{h \otimes h}"E";"X"};
\endxy
\] 
\end{enumerate}
\end{schol}

\subsection{Limits, filtered colimits and monadicity}
Let $\M$ be symmetric monoidal closed category, that is also complete and cocomplete. Recall that being symmetric closed implies that every functor $m \otimes -$ preserves colimits. The following proposition is classical, there are many references in the literature.
\begin{prop}
In the category $\com$ the following hold. 
\begin{enumerate}
\item Limits are computed in $\M$. In particular,  the forgetful functor $\Ub: \com \to \M$ preserves them.
\item Filtered colimits are also computed in $\M$ and $\Ub$ preserves them.
\item Coequalizers of $\Ub$-split pairs are computed in $\M$ and $\Ub$ preserves them.
\end{enumerate}
\end{prop}

The same proposition holds for the category $\coms$, as the reader will check.
\begin{prop}\label{closure-limit-filt}
In the category $\coms$ the following hold. 
\begin{enumerate}
\item Limits are computed in $ \Hom(\Phepiopl,\M)$, whence level-wise in $\M$. In particular,  the forgetful functor $\Ub: \coms \to \Hom(\Phepiop,\M)$ preserves them.
\item Filtered colimits are also computed in $ \Hom(\Phepiopl,\M)$ and $\Ub$ preserves them.
\item Coequalizers of $\Ub$-split pairs are computed in $ \Hom(\Phepiopl,\M)$ and $\Ub$ preserves them.
\end{enumerate}
In particular the category $\coms$ is complete and closed under filtered colimits.
\end{prop}

We want to show that the forgetful functor $\Ub:\coms \to \Hom(\Phepiopl,\M)$ a left adjoint $\Gamma:\Hom(\Phepiopl,\M) \to \coms$. And we want to show that $\Gamma$ doesn't change much the value at $\1$. 

\subsection{Existence of a left adjoint}
We will use  many times the main theorem in \cite{Adamek-Rosicky-reflection} to show that there is an abstract left adjoint $\Gamma:\Hom(\Phepiopl,\M) \to \coms$ to the forgetful functor $\Ub$, when $\M$ is locally presentable. For the reader's convenience we recall this theorem hereafter.
\begin{thm}[\cite{Adamek-Rosicky-reflection}]\label{adamek-refl}
Let $\mathscr{H}$ be a locally presentable category, and $\alpha$ a regular cardinal. Then each full subcategory of $\mathscr{H}$ closed under limits and $\alpha$-filtered colimits is reflective in $\mathscr{H}$.
\end{thm}

From now, we assume that $\M$ is a symmetric monoidal model category that is also  combinatorial and left proper. In particular $\M$ is locally presentable. It's classical that any diagram category like $\Hom(\Phepiopl,\M)$ is also locally presentable (see \cite{Adamek-Rosicky-loc-pres}).\\ 

We want to apply Theorem \ref{adamek-refl} for some category $\mathscr{H}$, but to do this properly we consider the following data.
\begin{df}
Define $\Lax(\Phepiop,\M)_{NA}$ as the category of nonassociative premonoids.
\begin{itemize}[label=$-$]
\item An object $\F$ is given by the same data as a nonsymmetric normal lax functor $$\F: \Phepiop \to \M,$$ without the axiom of associativity. This means that we only have laxity maps $$\F(\n) \ox \F(\m) \to \F(\n + \m)$$ that are compatible with the structure map $\n \to \n'$ and the addition of $\Phepiop$. The later simply says that the following commutes.
\begin{equation}
\xy
(0,18)*+{\F(\n) \otimes \F(\m)}="W";
(0,0)*+{\F(\n') \otimes \F(\m') }="X";
(40,0)*+{\F(\n'+\m')}="Y";
(40,18)*+{\F(\n+\m)}="E";
{\ar@{->}^-{\varphi}"X";"Y"};
{\ar@{->}^-{}"W";"X"};
{\ar@{->}^-{\varphi}"W";"E"};
{\ar@{->}^-{}"E";"Y"};
\endxy
\end{equation}
\item A morphism $\sigma: \F \to \G$ is given by the same data as a morphism in $\coms$. This means that $\sigma$ is a natural transformation such that the following commutes.
\begin{equation}
\xy
(0,18)*+{\F(\n) \otimes \F(\m)}="W";
(0,0)*+{\G(\n) \otimes \G(\m) }="X";
(40,0)*+{\G(\n+\m)}="Y";
(40,18)*+{\F(\n+\m)}="E";
{\ar@{->}^-{\varphi}"X";"Y"};
{\ar@{->}^-{\sigma \otimes \sigma}"W";"X"};
{\ar@{->}^-{\varphi}"W";"E"};
{\ar@{->}^-{\sigma}"E";"Y"};
\endxy
\end{equation} 
\end{itemize}
\end{df}

\begin{note}
In the monoidal category $\M=(\ul{M}, \otimes, I)$, there is a tautological monoid structure on the unit $I$ that corresponds to the isomorphism $I \otimes I \cong I$. In particular we can regard $I$ with this trivial monoid structure as an object of $\coms$ and $\Lax(\Phepiop,\M)_{NA}$. We will still denote this monoid by $I$.
\end{note}

\begin{df}
Define the category of nonassociative pointed premonoids as the comma category $$(I \downarrow \Lax(\Phepiop,\M)_{NA}).$$ 
In particular we have a chain of forgetful functors:
\begin{equation}
\coms \to (I \downarrow \Lax(\Phepiop,\M)_{NA}) \to \Lax(\Phepiop,\M)_{NA} \to \Hom(\Phepiopl,\M)
\end{equation}
\end{df}
The next proposition allows us to reduce the existence of the left adjoint $\Gamma$ to the existence of a left adjoint to the last functor in the chain:
$$\Lax(\Phepiop,\M)_{NA} \to \Hom(\Phepiopl,\M).$$ 

\begin{prop}\label{prop-reduction}
Let $\M$ be a symmetric monoidal closed category whose underlying category is locally presentable for a sufficiently large regular cardinal $\kappa$. With the previous notation, the following hold. 
\begin{enumerate}
\item The forgetful functor $\coms \to (I \downarrow \Lax(\Phepiop,\M)_{NA})$ identifies $\coms$ with a full subcategory of $(I \downarrow \Lax(\Phepiop,\M)_{NA})$ that is closed under limits and filtered colimits.
\item If $(I \downarrow \Lax(\Phepiop,\M)_{NA})$ is locally presentable then there is a left adjoint to this (inclusion) functor and $\coms$ is also locally presentable. In that case, there is a left adjoint 
$$\Gamma: \Hom(\Phepiopl,\M) \to \coms.$$
\item If $\Lax(\Phepiop,\M)_{NA})$ is locally presentable then so is $(I \downarrow \Lax(\Phepiop,\M)_{NA})$. 
\end{enumerate}
\end{prop}

\begin{proof}
The fact that $\coms$ corresponds to a full subcategory of $(I \downarrow \Lax(\Phepiop,\M)_{NA})$ follows from the definition. The closure under limits and filtered colimits follows from Proposition \ref{closure-limit-filt} which says that these two operations are created in $\Hom(\Phepiopl,\M)$. This gives Assertion $(1)$.\\

For the second assertion it suffices to take $\mathscr{H}=(I \downarrow \Lax(\Phepiop,\M)_{NA})$ and apply the main theorem in \cite{Adamek-Rosicky-reflection} mentioned before. If this happens then the induced adjunction is monadic by Beck monadicity (see \cite{Adamek_Rosicky_Vitale}). Indeed, the forgetful functor reflects isomorphisms and creates (whence preserves) coequalizer of reflexive pairs (also Proposition \ref{closure-limit-filt}). Then $\coms$ is equivalent to the category $\T$-alg of $\T$-algebras of the induced monad $\T$.\\

Now $\T$ is a monad defined on the locally presentable category  $(I \downarrow \Lax(\Phepiop,\M)_{NA})$ and $\T$ also preserves filtered colimits. Then following \cite[Remark 2.78]{Adamek-Rosicky-loc-pres} the category $\T$-alg, whence $\coms$, is also locally presentable.\\

Finally the forgetful functor $\Ub: \coms \to \Hom(\Phepiopl,\M)$ is a functor between locally presentable categories that preserves limits and filtered colimits. The adjoint functor theorem for locally presentable categories, gives the existence of a left adjoint $\Gamma$. This gives the second assertion.\\

Assertion $(3)$ is clear since any comma category of a locally presentable category is also locally presentable (see \cite{Adamek-Rosicky-loc-pres}). 
\end{proof}

Finally we have our reduction given by the following result.
\begin{lem}\label{lem-reduc-adj-nass}
Let $\M$ be a symmetric monoidal closed category that is also locally presentable. Then there is a left adjoint to the forgetful functor 
$$\Lax(\Phepiop,\M)_{NA} \to \Hom(\Phepiopl,\M).$$

The induced adjunction is monadic and $\Lax(\Phepiop,\M)_{NA}$ is also locally presentable.
\end{lem}

\begin{proof}
We postpone the proof to the Appendix. But the idea is that $\Lax(\Phepiop,\M)_{NA}$ is a category defined by algebraic equations with coefficients in a symmetric monoidal category $\M$ whose product $\otimes$ distributes over colimits.
\end{proof}
But for now we have a corollary that summarizes what we've just established. 

\begin{thm}\label{exist-left-adj}
Let $\M$ be a symmetric monoidal closed category that is also locally presentable. Then there is a left adjoint to the forgetful functor 
$$\coms \to \Hom(\Phepiopl,\M).$$

The induced adjunction is monadic and $\coms$ is also locally presentable.
\end{thm}
\begin{proof}
Just combine Proposition \ref{prop-reduction} and Lemma \ref{lem-reduc-adj-nass}.
\end{proof}

\section{Transferred model structure from $\M$ to $\coms$}\label{sec-transf}
The main purpose of this section is to show that there is a special model structure on $\coms$ in which a map $\sigma: \F \to \G$ is a weak equivalence (res. fibration) if the initial component $\F(\1) \to \G(\1)$ is a weak equivalence (resp. fibration) in $\M$.\\

The material of the preceding section gives us two choices to get this model structure. The first method is to put a model structure on  $\Hom(\Phepiopl,\M)$ and then transfer the model structure to $\coms$ using the well known lemma of \cite{Sch-Sh-Algebra-module} through the monadic adjunction:
 $$\Ub:\coms \leftrightarrows \Hom(\Phepiopl,\M):\Gamma.$$
 
 But this will be too long, so instead we will work directly with $\M$ using the fact that we have a left adjoint to the evaluation at $\1$:
 $$ \M \xrightarrow{\Fb^{\1}} \Hom(\Phepiopl,\M) \xrightarrow{\Gamma} \coms.$$ 
In this setting we will use the theory of \emph{right-induced} model structure (see \cite[Theorem 11.3.2]{Hirsch-model-loc}). Before doing this we need to outline some facts about the left adjoint $\Gamma \Fb^{\1}$.
\subsection{Properties of the left adjoint}

Let $\M=(\ul{M}, \ox, I)$ be as before and let $I \downarrow \M$ be the under category. Observe that we have a factorization of the evaluation at $\1$ as:
$$ \coms \xrightarrow{} (I \downarrow \M) \to \M.$$

The projection $(I \downarrow \M) \to \M$ has a left adjoint that takes $m \in \M$ to the canonical map $I \hookrightarrow (I \sqcup m)$ going to the coproduct:
 $$(I \hookrightarrow I \sqcup - ):\M \to (I \downarrow \M).$$ 
 In particular it's a left Quillen functor if we put on $(I \downarrow \M)$ the under model structure (see \cite{Hirsch-model-loc}, \cite{Hov-model}).
The following result is very important for the upcoming sections. It shows among other things the differences between co-Segal commutative premonoids and usual commutative monoids.

\begin{prop}\label{property-of-adj-m-to-coms}
Let $\M$ a symmetric monoidal closed category, that is also locally presentable. Then the following hold. 
\begin{enumerate}
\item The functor $\Ub: \coms \to (I \downarrow \M)$ is a left adjoint. That is, there is a functor $taut_u: (I \downarrow \M) \to \coms$ that is right adjoint to $\Ub$. 
\item The functor $\Ub: \coms \to (I \downarrow \M)$ preserves colimits, in particular it preserves pushouts.
\item The composite of left adjoints $\M \xrightarrow{\Gamma\Fb^{1}} \coms \xrightarrow{\Ub} (I \downarrow \M)$ is a left adjoint to the projection $(I \downarrow \M) \to \M$. By uniqueness of the adjoint this composite is isomorphic to the functor 
$$(I \hookrightarrow I \sqcup - ):\M \to (I \downarrow \M).$$
\item The composite $\M \xrightarrow{\Gamma\Fb^{1}} \coms \xrightarrow{\Ub} (I \downarrow \M) \to \M$ is isomorphic to the functor 
$$(I \sqcup -): \M \to \M.$$
\item Assume that $\M$ is also a model category. Let $f: U \to V$ be a (trivial) cofibration in $\M$ and let $\Gamma\Fb^{1}_{f}: \Gamma\Fb^{1}_{U} \to \Gamma\Fb^{1}_{V}$ be it's image in $\coms$. Then the component of $\Gamma\Fb^{1}_{f}$ at $\1$ is isomorphic to the coproduct of $\Id_I$ and $f$:
$$\Id_I \sqcup f: I \sqcup U \to I \sqcup V.$$
In particular it's a (trivial) cofibration in $\M$.
\end{enumerate}
\end{prop}

\begin{proof}
Each assertion follows one after the other and we remember that left adjoints preserve all kind of colimits. Therefore it suffices to show that there is a right adjoint $taut_u$. For this let $\ast$ be the terminal object of $\M$.\\

Observe that there is a unique commutative monoid structure on $\ast$, where the multiplication $\ast \otimes \ast \to \ast$ is the unique one. Let's denote this commutative monoid by $[\ast]$ and let's regard it as an object of $\coms$ through the inclusion 
$\com \hookrightarrow \coms$.\\

Let $f:I \to m$ be any object of $(I \downarrow \M)$  and let $h: m \to \ast$ be the unique map in $\M$. Define $taut_u(f)=h^\star [\ast]$ as the object of $\coms$ obtained with the construction described in Scholium \ref{schol1} applied to $[\ast]$ with respect to the map $h: m \to \ast$.\\

We leave the reader to check that this is indeed a right adjoint to 
$$\coms \to (I \downarrow \M).$$
\end{proof}

As a corollary  we get the following lemma.
\begin{lem}\label{lem-important-pushout}
Let $\M$ be a symmetric monoidal model category that is combinatorial and left proper. 
\begin{enumerate}
\item  The functor $(I \downarrow \M) \to \M$ creates (hence preserves) pushouts. 
\item The functor $\Ev_1: \coms \to (I \downarrow \M) \to \M$ preserves pushouts. In particular pushouts in $\coms$ are taken level-wise at the value of $\1$.
\item Let $f:U \to V$ be a morphism in $\M$ and let $D$ be a pushout square in $\coms$:
 \[
 \xy
(0,18)*+{\Gamma\Fb^{1}_{U}}="W";
(0,0)*+{\Gamma\Fb^{1}_{V}}="X";
(30,0)*+{\Ba}="Y";
(30,18)*+{\Aa}="E";
{\ar@{->}^-{\ol{\sigma}}"X";"Y"};
{\ar@{->}^-{\Gamma\Fb^{1}_{f}}"W";"X"};
{\ar@{->}^-{\sigma}"W";"E"};
{\ar@{->}^-{\varepsilon}"E";"Y"};
\endxy
\]
Then the image of $D$ by the evaluation at $\1$ is the pushout square in $\M$:
 \[
 \xy
(0,18)*+{I \sqcup U}="W";
(0,0)*+{I \sqcup V}="X";
(30,0)*+{\Ba(\1)}="Y";
(30,18)*+{\Aa(\1)}="E";
{\ar@{->}^-{\ol{\sigma}_{\1}}"X";"Y"};
{\ar@{->}^-{\Id_I \sqcup f}"W";"X"};
{\ar@{->}^-{\sigma_{\1}}"W";"E"};
{\ar@{->}^-{\varepsilon}"E";"Y"};
\endxy
\]
In particular:
\begin{itemize}[label=$-$]
\item The map $\varepsilon_{\1}: \Aa(\1) \to \Ba(\1)$ is a trivial cofibration if $f$ is;
\item Since $\M$ is left proper, the map $\ol{\sigma}_{\1}$ is a weak equivalence if $f$ is a cofibration and if $\sigma_\1$ is a weak equivalence. 
\end{itemize}
\end{enumerate}
\end{lem}

\begin{proof}
Clear.
\end{proof}

\subsection{The easy homotopy theory on $\coms$}
The material of the preceding sections was a preparation for our first theorem. A direct application of \cite[Theorem 11.3.2]{Hirsch-model-loc}  provides a right-induced model structure on $\coms$ through the adjunction 
$$\Ev_\1: \coms \leftrightarrows \M: \Gamma \Fb^{\1}.$$

\begin{thm}\label{easy-model-coms}
Let $\M=(\ul{M}, \otimes, I)$ be symmetric monoidal model category that is also combinatorial and left proper. Let $\I$ and $\Ja$ be respectively the set of generating cofibrations and the set of generating trivial cofibrations. 
\begin{enumerate}
\item Then there is a combinatorial and left proper model structure on $\coms$ which may be described as follows.
\begin{itemize}[label=$-$]
\item A map $\sigma :\F \to \G$ is a weak equivalence (resp. fibration) if $$\Ev_\1(\sigma): \F(\1) \to \G(\1),$$ is a weak equivalence (resp. fibration) in $\M$.
\item A map  $\sigma :\F \to \G$ is a cofibration if it has the LLP against any map that is a weak equivalence and a fibration. 
\end{itemize}
\item This model structure is also left proper.
\item The set $\Gamma\Fb^{\1}(\I)$ is a set of generating cofibrations and the set $\Gamma\Fb^{\1}(\Ja)$ is a set of generating  trivial cofibrations.

\item We will denote by $\comse$ this model category. The  adjunction
$$ \Ev_\1: \comse \rightleftarrows  \M: \Gamma\Fb^{\1}$$
is a Quillen adjunction where $\Gamma\Fb^{\1}$ is left Quillen and $\Ev_\1$ is right Quillen.
\end{enumerate} 
\end{thm}

\begin{proof}
Follows from Lemma \ref{lem-important-pushout} and \cite[Theorem 11.3.2]{Hirsch-model-loc}. Finally $\coms$ is also locally presentable.
\end{proof}

\begin{term}
A weak equivalence in $\comse$ will be called \emph{easy weak equivalence}. This is the reason we included the letter `e' as a subscript.
\end{term}
We have a similar result for the category $\Hom(\Phepiopl, \M)$ and the adjunction 
$$\Ev_\1: \Hom(\Phepiopl, \M) \leftrightarrows \M: \Fb^{\1}.$$
We include it here for completeness.
\begin{thm}\label{easy-model-grahpcoms}
Let $\M=(\ul{M}, \otimes, I)$ be symmetric monoidal model category that is also combinatorial and left proper. Let $\I$ and $\Ja$ be respectively the set of generating cofibrations and the set of generating trivial cofibrations. 
\begin{enumerate}
\item Then there is a combinatorial and left proper model structure on $\Hom(\Phepiopl, \M)$ which may be described as follows.
\begin{itemize}[label=$-$]
\item A map $\sigma :\F \to \G$ is a weak equivalence (resp. fibration) if $$\Ev_\1(\sigma): \F(\1) \to \G(\1),$$ is a weak equivalence (resp. fibration) in $\M$.
\item A map  $\sigma :\F \to \G$ is a cofibration if it has the LLP against any map that is a weak equivalence and a fibration. 
\end{itemize}
\item This model structure is also left proper.

\item The sets $\Fb^{\1}(\I)$ is a set of generating cofibrations and the set $\Fb^{\1}(\Ja)$ is a set of generating  trivial cofibrations.

\item We will denote by $\Hom(\Phepiopl, \M)_e$ this model category. The  adjunction
$$ \Ev_\1: \Hom(\Phepiopl, \M) \rightleftarrows  \M: \Fb^{\1}$$
is a Quillen adjunction where $\Fb^{\1}$ is left Quillen and $\Ev_\1$ is right Quillen.
\item We have a chain of Quillen adjunctions, in which $\Gamma$ is left Quillen.
$$\coms \leftrightarrows \Hom(\Phepiopl, \M) \rightleftarrows  \M.$$
\end{enumerate} 
\end{thm}
\subsection{First Quillen adjunction between strict and weak monoids}
We isolated here the comparison between $\com$ and $\coms$. On $\com$ we will consider the model structure where weak equivalences and fibrations are  such maps between the underlying objects in $\M$ (see \cite{White_com}). We will refer to this model structure as the \emph{natural} model structure. 
\begin{thm}
Let $\M=(\ul{M}, \otimes, I)$ be symmetric monoidal closed category that is also locally presentable. Then the following hold. 
\begin{enumerate}
\item The inclusion $\iota: \com \hookrightarrow \coms$ exhibits $\com$ as a full subcategory closed under limits and filtered colimits. In particular there is a left adjoint (reflection) to the inclusion $\iota$:
$$|-|: \coms \to \com.$$

\item Assume that $\M$ is a combinatorial monoidal model category. If the natural model structure on $\com$ exists then we have a Quillen adjunction
$$ |-|: \comse \leftrightarrows \com : \iota,$$
 in which $\iota: \com \to \comse$ is right Quillen. 
\end{enumerate}
\end{thm}

\begin{proof}
The first assertion is clear since limits and filtered colimits in $\com$ and $\coms$ are computed in $\M$. Then it suffices to apply Theorem \ref{adamek-refl} or alternatively the adjoint theorem for locally presentable categories.\\
 
The second assertion is clear since $\iota$ preserves (and reflects) fibrations and trivial fibrations (as well as weak equivalences).
\end{proof}

\subsection{Left adjoint from $\Hom(\Phepiopl, \M)$ to $\coms$}
We saw in Proposition \ref{property-of-adj-m-to-coms} that the composite  of left adjoints $$\M \xrightarrow{\Gamma\Fb^{1}} \coms \xrightarrow{\Ub} (I \downarrow \M) \to \M$$ is isomorphic to the functor 
$$(I \sqcup -): \M \to \M.$$

This result implicitly informs us that given $\F \in \Hom(\Phepiopl, \M)$, then $\Gamma (\F)(\1)$ is simply the coproduct $ I \sqcup \F(\1)$. This is the major difference between classical commutative monoids and co-Segal monoids. For classical monoids, the free monoids structure is concentrated in one term, whereas here, the free monoid is built (slowly) within the terms $\Gamma (\F)(\n)$ of higher degree.\\

 We  confirm this `officially' with the next result but for the moment we need to outline some basic facts.

\begin{rmk}
\begin{enumerate}

\item Let $\uc$ be the unit category and let $i_1: \uc \to \Phepiopl$ be the functor that selects the object $\1 \in \Phepiopl$. Then the reader can easily see that:
\begin{itemize}[label=$-$]
\item The evaluation functor $\Ev_\1: \Hom(\Phepiopl, \M) \to \M$ is nothing but the pullback functor 
$$ i_\1^{\star}: \Hom(\Phepiopl, \M) \to \Hom(\uc, \M) \simeq \M .$$
\item And this pullback functor has  a left adjoint given by the left Kan extension $Lan_{i_\1}$,  as well as a right adjoint given by the right Kan extension $Ran_{i_\1}$ (see \cite{Mac}).
\item With a little thinking it's not hard to see that, given $m \in \M$, $Ran_{i_\1}(m)$ is the underlying diagram of the right adjoint $taut_u(m)$ mentioned in Proposition \ref{property-of-adj-m-to-coms}.
\end{itemize}
\item It turns out that $\Ev_\1: \Hom(\Phepiopl, \M) \to \M$ is simultaneously a left adjoint and a right adjoint. 
\item Using the left adjoint $\M \to (I \downarrow \M)$ we get a left adjoint by composition
$$\Hom(\Phepiopl, \M) \xrightarrow{\Ev_\1} \M \xrightarrow{ [I \to (I \sqcup -)]}(I \downarrow \M).$$
As usual left adjoints preserve all kind of colimits, so this functor certainly preserves pushout. 
\end{enumerate}
\end{rmk} 

A direct consequence of this remark is the Proposition below. It will play an important role in a moment. 

\begin{prop}
In a monoidal category $\M$ satisfying the previous hypotheses, the following hold.
\begin{enumerate}
\item The  composite of left adjoints
$$ \Hom(\Phepiopl, \M) \xrightarrow{\Gamma} \coms \xrightarrow{\Ev_{\1}} (I \downarrow \M),$$ is isomorphic to the other composite of left adjoints:
$$\Hom(\Phepiopl, \M) \xrightarrow{\Ev_\1} \M \xrightarrow{ [I \to (I \sqcup -)]}(I \downarrow \M).$$
\item Let $\sigma: \F \to \G$ be a natural transformation in $\Hom(\Phepiopl, \M)$ and let $\Gamma(\sigma)$ be its image in $\coms$. Then the component of  $\Gamma(\sigma)$ at $\1$ is isomorphic to the coproduct:
$$\Id_I \sqcup \sigma_\1: I \sqcup \F(\1) \to I \sqcup \G(\1).$$
In particular $\Gamma(\sigma)_\1$ is a (trivial) cofibration if $\sigma_\1$ is a (trivial) cofibration.
\end{enumerate}
\end{prop}
\begin{proof}
Both functors are left adjoint to the composite $$(I \downarrow \M) \xrightarrow{taut_u} \coms  \xrightarrow{\Ub} \Hom(\Phepiopl, \M).$$
The uniqueness of the left adjoint gives the first assertion. The second assertion is a consequence of the first. Finally in any model category the class of (trivial) cofibrations is closed under coproduct.
\end{proof}

We close this subsection with a general version of Lemma \ref{lem-important-pushout}. 

\begin{lem}\label{lem-important-pushout-general}
Let $\M$ be a symmetric monoidal model category that is combinatorial and left proper. 
\begin{enumerate}
\item  The functor $(I \downarrow \M) \to \M$ creates (hence preserves) pushouts. 
\item The functor $\Ev_1: \coms \to (I \downarrow \M) \to \M$ preserves pushouts. In particular pushouts in $\coms$ are taken level-wise at the value of $\1$.
\item Let $\theta: \F \to \G$ be a map in $\coms$ and let $D$ be a pushout square in $\coms$:
 \[
 \xy
(0,18)*+{\F}="W";
(0,0)*+{ \G}="X";
(30,0)*+{\Ba}="Y";
(30,18)*+{\Aa}="E";
{\ar@{->}^-{\ol{\sigma}}"X";"Y"};
{\ar@{->}^-{\theta}"W";"X"};
{\ar@{->}^-{\sigma}"W";"E"};
{\ar@{->}^-{\varepsilon}"E";"Y"};
\endxy
\]
Then the image of $D$ by the evaluation at $\1$ is the pushout square in $\M$:
 \[
 \xy
(0,18)*+{\F(\1)}="W";
(0,0)*+{ \G(\1)}="X";
(30,0)*+{\Ba(\1)}="Y";
(30,18)*+{\Aa(\1)}="E";
{\ar@{->}^-{\ol{\sigma}_{\1}}"X";"Y"};
{\ar@{->}^-{\theta_\1}"W";"X"};
{\ar@{->}^-{\sigma_{\1}}"W";"E"};
{\ar@{->}^-{\varepsilon}"E";"Y"};
\endxy
\]
In particular:
\begin{itemize}[label=$-$]
\item The map $\varepsilon_{\1}: \Aa(\1) \to \Ba(\1)$ is a (trivial) cofibration if $\theta_\1$ is a (trivial) cofibration;
\item Since $\M$ is left proper, the map $\ol{\sigma}_{\1}$ is a weak equivalence if $\theta_1$ is a cofibration and if $\sigma_\1$ is a weak equivalence. 
\end{itemize}
\end{enumerate}
\end{lem}

\begin{proof}
Clear.
\end{proof}

\begin{note}
We will use this lemma later to show that we can still have a model structure on $\coms$ with the same weak equivalences but with more cofibrations. 
\end{note}

\section{Localizing sets}\label{section-loc-set}

\subsection{From the arrow category $\Ar(\M)$ to $\coms$}
\begin{warn}
We would like to warn the reader about our notation for the \emph{walking-morphism category}. Although it's natural to denote this category by $[1]= \{0 \to 1\}$ or $[2]=\{1 \to 2\}$ we will \ul{not} use this notation. The reason being that it might create a confusion with our notation for the objects $\1,\2$ of $\Phi$. Therefore we will denote this category by $\Un$.

\textsl{From now we will use the notation $\Un=[\ap \to \bp]$ for the walking morphism category}.
\end{warn}

By definition, any morphism $f: a \to b$ in a category $\B$ can be identified with a functor $$f:\Un \to \B$$ given by $f(-)=a$ and $f(+)=b$. This justifies the notation $\Ar(\B)=\B^{\Un}$.
\begin{nota}
\begin{enumerate}
\item Let $\n >\1$ be an object of $\Phepiopl$ and let $u_\n: \1 \to \n$ be the unique morphism therein. We can identity $u_\n$ with a functor $u_\n:\Un \to \Phepiopl$. Pulling back along $u_\n$ is just the evaluation at $u_\n$. We will denote by $\Ev_{u_\n}: \coms \to \M^{\Un}$
the composite $$\coms \xrightarrow{\Ub} \Hom(\Phepiopl,\M) \xrightarrow{u_\n^\star} \M^{\Un}.$$
\item Let $\iota_\M:\M \to \M^{\Un}$ be the natural embedding that takes an object $m$ to $\Id_{m}$; it takes a morphism $f: m \to m'$ to the morphism $[f]: \Id_m \to \Id_{m'}$ of $\M^{\Un}$, whose components are both equal to $f$.
\item Finally let $\Ev_+: \M^{\Un} \to \M$ be the `target functor' that takes $f:m \to m'$ to its target $m'$. It's not hard to see that $\Ev_+$ is a left adjoint to $\iota_\M$.  
\end{enumerate}
\end{nota}

\begin{prop}\label{prop-av}
Let $\n >\1$ be an object of $\Phepiopl$, and let 
$u_\n: \1 \to \n$ be the unique morphism therein. Then the following hold.
\begin{enumerate}
\item The evaluation $\Ev_{u_\n}:\coms \to \M^{\Un}$ has a left adjoint
$$\Psi_\n: \M^{\Un} \to \coms.$$ 
\item If $\chi: f \to g$ is a morphism in $\M^{\Un}$ such that the component $\chi_{\ap}: f(\ap) \to g(\ap)$ is a (trivial) cofibration then the component $\Psi_\n(\chi)_{\1}: \Psi_\n(f)(\1) \to \Psi_\n(g)(\1)$ is also a (trivial) cofibration in $\M$.
\end{enumerate}
\end{prop}
\begin{proof}
Let $Lan_{u_\n}: \M^{\Un} \to \Hom(\Phepiopl, \M)$ be the left Kan extension along $u_\n$. Then $\Psi_\n$ is just the composite of left adjoints:
$$ \M^{\Un} \xrightarrow{Lan_{u_\n}} \Hom(\Phepiopl, \M)\xrightarrow{\Gamma} \coms.$$
This gives the first assertion. As for the second, it suffices to write the formula for $Lan_{u_\n}$. 

Indeed, since $\1$ is initial in $\Phepiopl$, it's not hard to see that  component of $Lan_{u_\n}(\sigma)$ at $\1$ is isomorphic to the component $\chi_{-}: f(\ap) \to g(\ap)$. Now thanks to Lemma \ref{lem-important-pushout-general}, we know that $\Gamma$  changes this component to the coproduct
$$\Id \sqcup \chi_{-}: I \sqcup f(-) \to I \sqcup g(-).$$
\end{proof}

Recall that the inclusion $\iota: \com \hookrightarrow \coms$ is a right adjoint and we've denoted by $|-|:\coms \to \com$ the corresponding left adjoint. In particular we have a left adjoint by composition:
$$\M^{\Un} \xrightarrow{\Psi_\n} \coms \xrightarrow{|-|}\com.$$

The next lemma gives a useful information about this left adjoint. Indeed this functor is left adjoint to the usual forgetful functor $\com \to \M$ followed by the embedding $\iota_\M: \M \hookrightarrow \M^{\Un}$.  

Recall that there is a classical left adjoint $Free: \M \to \com$ that extends to $\Gamma: \Hom(\Phepiopl,\M) \to \coms$ in co-Segal settings. $Free$ factors through $\Gamma$ as:
$$\M \xrightarrow{\Fb^{\1}}  \Hom(\Phepiopl,\M) \xrightarrow{\Gamma} \coms \xrightarrow{|-|} \com.$$

\begin{lem}
Let $\M$ be as previously. Then the commutative square on the left is a square of left adjoints and the commutative square on the right is the  square corresponding to the respective right adjoints.
\[
\xy
(0,18)*+{\M^{\Un}}="W";
(0,0)*+{\coms}="X";
(30,0)*+{\com}="Y";
(30,18)*+{\M}="E";
{\ar@{->}^-{|-|}"X";"Y"};
{\ar@{->}^-{\Psi_\n}"W";"X"};
{\ar@{->}^-{\Ev_+}"W";"E"};
{\ar@{->}^-{Free}"E";"Y"};
\endxy
\hspace{1.5in}
\xy
(0,18)*+{\M^{\Un}}="W";
(0,0)*+{\coms}="X";
(30,0)*+{\com}="Y";
(30,18)*+{\M}="E";
{\ar@{<-}^-{\iota}"X";"Y"};
{\ar@{<-}^-{\Ev_{u_\n}}"W";"X"};
{\ar@{<-}^-{\iota_\M}"W";"E"};
{\ar@{<-}^-{\Ub}"E";"Y"};
\endxy
\]

In particular if $\chi: f \to g$ is a morphism in $\M^{\Un}$ such that the component $\chi_+: f(+) \to g(+)$ is an isomorphism in $\M$, then the morphism of usual commutative monoids
$$|\Psi_\n(\chi)|: |\Psi_\n(f)| \to  |\Psi_\n(g)|$$
is an isomorphism.
\end{lem}

\begin{proof}
Indeed the functor $Free: \M \to \com$ preserves isomorphisms as any functor and we have an equality $|\Psi_\n(\chi)|= Free \circ \Ev_+(\chi)= Free( \chi_+).$
\end{proof}
\subsection{Definition of the localizing morphisms}
\begin{nota}\label{notation-set-localization}
If $\alpha: U \to V$ is a morphism of $\M$, we will denote by 
$$\alpha_{\downarrow_{\Id_V}}: \alpha \to \Id_V,$$ 
the morphism in the arrow category $\M^{\Un}$ which is identified with the following commutative square. 
\[
\xy
(0,18)*+{U}="W";
(0,0)*+{V}="X";
(30,0)*+{V}="Y";
(30,18)*+{V}="E";
{\ar@{->}^-{\Id}"X";"Y"};
{\ar@{->}^-{\alpha}"W";"X"};
{\ar@{->}^-{\alpha}"W";"E"};
{\ar@{->}^-{\Id}"E";"Y"};
\endxy
\]
\end{nota}
\begin{rmk}\label{rmk-projec-com}
Following our previous notation, the component $(\av)_-$ is exactly the map $\alpha:U \to V$ and the component $(\av)_+$ is the identity $\Id_V$ (which is a wonderful isomorphism). Then by the previous results we know that the image $|\Psi_\n(\av)|$ is an isomorphism of (free) commutative monoids.
\end{rmk}


\begin{df}
 Let $\I$ be set of generating cofibrations of $\M$.
\begin{enumerate}
\item Define the \emph{localizing set} for $\coms$ as
$$\kb(\I):= \{\coprod_{\n \geq 2} \{\Psi_\n(\av); \quad \alpha \in \I  \} \}.$$
\item Let $\ast$ be the coinitial (or terminal) object of $\coms$ and let  $\sigma$ be map in $\coms$. Say that an object $\F \in \msxsu$ is $\sigma$-injective if the unique map $\F \to \ast$ has the RLP with respect to $\sigma$. 
\item Say that $\F$ is $\kb(\I)$-injective if $\F$ is $\sigma$-injective for all $\sigma \in \kb(\I)$.
\end{enumerate}
\end{df}

Before going further we have the following result. 

\begin{prop}\label{prop-comp-kbi}
With the previous notation the following assertions hold. 
\begin{enumerate}
\item If $\sigma=\Psi_\n(\av)$ is an element of $\kb(\I)$ then the component 
$\sigma_\1$ is a cofibration in $\M$ and is isomorphic to the coproduct 
$$\Id_I \sqcup \alpha.$$ 
\item The image of any $\sigma \in \kb(\I)$ by the left adjoint $|-|: \coms \to \com$ is an isomorphism. 
\end{enumerate}
\end{prop}
\begin{proof}
The first assertion is given by Proposition \ref{prop-av}. The second assertion is just the content of Remark \ref{rmk-projec-com}.
\end{proof}

\subsection{Characteristics of the set $\kb(\I)$}
The following proposition is not hard, one simply needs to write down everything.
\begin{prop}\label{lem-lifting-inject}
Let $\theta=(f,g): \alpha \to p$ be a morphism in $\M^{\Un}$ which is represented by the following commutative square.

\[
\xy
(0,20)*+{U}="A";
(20,20)*+{X}="B";
(0,0)*+{V}="C";
(20,0)*+{Y}="D";
{\ar@{->}^{f}"A";"B"};
{\ar@{->}_{\alpha}"A";"C"};
{\ar@{->}^{p}"B";"D"};
{\ar@{->}^{g}"C";"D"};
\endxy
\]  

Then the following are equivalent. 
\begin{itemize}[label=$-$]
\item There is a lifting in the commutative square above i.e there exists $k: V \to X$  such that: $k \circ \alpha =f$, $p \circ k=g$.
\item There is a lifting in the following square of $\M^{\Un}$.
\[
\xy
(0,20)*+{\alpha}="A";
(20,20)*+{p}="B";
(0,0)*+{\Id_V}="C";
(20,0)*+{\ast}="D";
{\ar@{->}^{\theta}"A";"B"};
{\ar@{->}_{\av}"A";"C"};
{\ar@{->}_{}"C";"D"};
{\ar@{->}_{}"B";"D"};
\endxy
\]  

That is, there exists $\beta=(k,l): \Id_V \to p$ such that $\beta \circ \av= \theta$. 
\end{itemize}
\end{prop}

Thank to this proposition and the fact that trivial fibrations in $\M$ are the $\I$-injective maps, we can establish by adjointness the following result. 
\begin{lem}\label{k-inj-cosegal}
Let $\F$ be an object of $\coms$. Then with the previous notation, the following hold.
\begin{enumerate}
\item $\F$ is $\kb(\I)$-injective if and only if for every $\n \geq \2$ the map $$\F(u_\n): \F(\1) \to \F(\n),$$ is a trivial fibration in $\M$. In particular if $\F$ is  $\kb(\I)$-injective, then $\F$ is a co-Segal commutative monoid.
\item Every  a strict commutative monoid $\F \in \com$  is $\kb(\I)$-injective.
\end{enumerate}
\end{lem}

\begin{proof}
If $\F$ is $\kb(\I)$-injective, by definition, $\F$ is $\Psi_\n(\av)$-injective for all generating cofibration $\alpha$ in $\M$. And by adjointness we find that $\F(u_\n)$ is $\av$-injective. Thanks to the previous proposition, this is equivalent to saying that any lifting problem defined by $\alpha$ and $\F(u_\n)$ has a solution.

 Consequently $\F$ is $\kb(\I)$-injective if and only if  $\F(u_\n)$ has the RLP with respect to all maps in $\I$, if and only if  $\F(u_\n)$ is a trivial fibration as claimed. This proves Assertion $(1)$. 

Assertion $(2)$ is a corollary of Assertion $(1)$ since strict commutative monoids are the constant lax diagrams.  Therefore $\F(u_n)$ is an identity which is a trivial fibration.
\end{proof}

\section{Subcategory of $2$-constant  commutative premonoids}\label{section-2-constant}
\begin{note}
The discussion that follows is motivated by our desire to analyze the unit of the adjunction 
$|-|:\coms \leftrightarrows \com: \iota$. 
\end{note}

We take a moment to outline an important class of commutative premonoids that will be needed later. Let  $\Phepiop_{\geq \2} \subset \Phepi$ the full  subcategory formed by the  objects $\n \geq \2$. 
\begin{df}
Say that a commutative premonoid  $\F \in \coms$ is \ul{$2$-constant} if the restriction of $\F$ to $\Phepiop_{\geq \2}$ 
$$\F_{\geq \2}: \Phepiop_{\geq \2} \to \M,$$
is a constant functor.
\end{df}
\begin{rmk}
\begin{enumerate}
\item It follows from the definition that any $2$-constant premonoid has an underlying nonunital commutative monoid $\F_{\geq \2}$.  This premonoid $\F_{\geq 2}$ inherits of the pseudo-unit element maps $e: I \to \F(\2)$ but  it doesn't necessarily satisfy the usual unitality conditions for strict commutative monoid. We may change our definition of unital premonoid so that $\F_{\geq 2}$ becomes automatically a usual commutative monoid but this won't change the outcome of the homotopy theory.
\item  The $2$-constant commutative premonoids we will consider in a moment, have the property that $\F_{\geq 2}$ is already a strict commutative monoid i.e, an object of $\com$.
\end{enumerate}
\end{rmk}

\begin{df}\label{df-perfect-2-unit}
Say that a  $2$-constant commutative premonoid $\F$ is \emph{perfectly} $2$-constant if  $\F_{\geq 2}$ is a usual strict commutative monoid.
\end{df}
\begin{warn}
From now on, when we say $2$-constant we mean perfectly $2$-constant. This is for simplicity only.
\end{warn}
\subsection{Associated $2$-constant commutative premonoid}
Let's consider again the adjunction $|-|:\coms \leftrightarrows \com: \iota$. 
\begin{df}\label{rmk-associated-deux-constant}
Let $\F$ be a commutative premonoid and let $\eta: \F \to \iota(|\F|)$ be the unity of the adjunction. Let $h: \F(\1) \to \iota(|\F|)(\1)$ be the component of this map  at the initial entry $\1$. 
\begin{enumerate}
\item Define the associated $2$-constant commutative premonoid of $\F$ to be commutative premonoid $h^\star \iota(|\F|)$ obtained using the construction described in Scholium \ref{schol1} applied to $\iota(|\F|)$ with respect to the maps $h$.  In particular  $[h^\star \iota(|\F|)]_{\geq 2}= \iota(|\F|)$ is a usual strict commutative monoid. 
\item Define the fundamental factorization for $\eta: \F \to \iota(|\F|)$ as:
$$\F \xrightarrow{\rho} h^\star \iota(|\F|) \xrightarrow{\epsilon} \iota(|\F|).$$
The map $\F \xrightarrow{\rho} h^\star \iota(|\F|) $ is the identity at the initial entry $\1$ and the map $ h^\star\iota(|\F|) \xrightarrow{\epsilon} \iota(|\F|)$ is the identity everywhere outside the initial entry.
\end{enumerate}
\end{df}

\begin{prop}\label{prop-equiv-constant}
With the previous notation, the following hold.
\begin{enumerate}
\item The map $\F \xrightarrow{\rho} h^\star \iota(|\F|)$ is an easy weak equivalence i.e, a weak equivalence in $\comse$.
\item Let $ L: \comse \to \Ba$ be a functor that takes easy weak equivalences to isomorphisms in $\Ba$. Then $L(\F \xrightarrow{\eta} \iota(|\F|))$ is an isomorphism in $\Ba$ if and only if $ L(h^\star\iota(|\F|) \xrightarrow{\epsilon} \iota(|\F|))$ is an isomorphism in $\Ba$
\end{enumerate}
\end{prop}
\begin{proof}
The component of $\rho$ at $\1$ is the identity which is a  weak equivalence, this gives Assertion $(1)$. Assertion $(2)$ is a consequence of Assertion $(1)$ together with the fact that isomorphisms in any category $\Ba$ have the $3$-for-$2$ property.
\end{proof}
\subsection{The small object argument for $2$-constant premonoids}
In the following we are interested in factoring the map $h^\star\iota(|\F|) \xrightarrow{\epsilon} \iota(|\F|)$ using the small object argument with respect to a subset of the localizing set $\kb(\I)$. We refer the reader to \cite{Dwyer_Spalinski}, \cite{Hov-model} for a detailed account on the small object argument. The idea amounts to take sequentially pushouts of coproduct of maps in $\kb(\I)$.

 It is then important  to have a careful analysis of such pushout. We start below with a proposition that allows us to reduce our analysis to the case of a pushout of a single map in $\kb(\I)$.
\begin{lem}\label{coproduc-pushout}
Let $\B$ be a category with all small colimits and let $f:A \to E$ be the pushout of the coproduct $\coprod_{k \in K} C_k \xrightarrow{\coprod g_k} \coprod_{k \in K}D_k$:
\[
\xy
(0,18)*+{\coprod_{k \in K} C_k}="W";
(0,0)*+{\coprod_{k \in K} D_k}="X";
(30,0)*+{E}="Y";
(30,18)*+{A}="E";
{\ar@{->}^-{p}"X";"Y"};
{\ar@{->}^-{\coprod g_k}"W";"X"};
{\ar@{->}^-{q}"W";"E"};
{\ar@{->}^-{f}"E";"Y"};
\endxy
\]

For every $k \in K$, let $f_k :A \to E_k$ be the pushout of $g_k$ along the attaching map $C_k \xhookrightarrow{i_k}\coprod_{k \in K} C_k \xrightarrow{q}A$:
\[
\xy
(0,18)*+{C_k}="W";
(15,18)*+{\coprod _k C_k}="Z";
(0,0)*+{D_k}="X";
(30,0)*+{E_k}="Y";
(30,18)*+{A}="E";
{\ar@{->}^-{}"X";"Y"};
{\ar@{->}^-{g_k}"W";"X"};
{\ar@{->}^-{}"W";"Z"};
{\ar@{->}^-{}"Z";"E"};
{\ar@{->}^-{f_k}"E";"Y"};
\endxy
\]

Let $O$ be the colimit of the wide pushout $\{A \xrightarrow{f_k} E_k \}_{k\in K}$ and let $\delta:A \to O$ be the canonical map going to the colimit.\\

Then we have an isomorphism $f \cong \delta$ in the comma category $(A \downarrow \B)$. In particular we have an isomorphism $O \cong E$ in $\B$.
\end{lem}
\begin{proof}
Let $\tau_k:E_k \to O$ be the canonical map going to the colimit of the wide pushout. By definition of the colimit we have an equality $\delta= \tau_k \circ f_k$. Using $\tau_k$ we can extend the pushout square defining $f_k$  to have the following commutative square.
\[
\xy
(0,18)*+{C_k}="W";
(15,18)*+{\coprod _k C_k}="Z";
(0,0)*+{D_k}="X";
(30,0)*+{O}="Y";
(30,18)*+{A}="E";
{\ar@{->}^-{}"X";"Y"};
{\ar@{->}^-{g_k}"W";"X"};
{\ar@{->}^-{}"W";"Z"};
{\ar@{->}^-{}"Z";"E"};
{\ar@{->}^-{\delta}"E";"Y"};
\endxy
\]
Let's regard this commutative square as a map $g_k \to \delta$ in the arrow category $\B^{[1]}$. The universal property of the coproduct implies that there is a unique induced map 
$ \coprod_k g_k \to \delta$ satisfying the usual factorizations.  The later map represents a commutative square in $\B$ as:
\[
\xy
(0,18)*+{\coprod_{k \in K} C_k}="W";
(0,0)*+{\coprod_{k \in K} D_k}="X";
(30,0)*+{O}="Y";
(30,18)*+{A}="E";
{\ar@{->}^-{}"X";"Y"};
{\ar@{->}^-{\coprod g_k}"W";"X"};
{\ar@{->}^-{}"W";"E"};
{\ar@{->}^-{\delta}"E";"Y"};
\endxy
\]

But since coproducts in $\B^{[1]}$ are taking point-wise, it's not hard to see that the attaching map $\coprod_{k \in K} C_k \to A$ is exactly the map $q$ in the original diagram.\\

We leave the reader to check that this commutative square is the universal  pushout square. That is, the object $O$ equipped with the map $\delta$ and the other one, satisfies the universal property of the pushout of $\coprod_k g_k$ along the attaching map $q$.
\end{proof}

\subsection{\emph{`co-Segalification'} for $2$-constant commutative premonoids}
If we want to define a functor $\Sim$ that takes a $2$-constant commutative premonoid $\F$ to a commutative premonoid that satisfies the co-Segal conditions, the natural thing to do is to factor the map $\F(\1) \to |\F|(\1)$ as a cofibration followed by a trivial fibration:
$$\F(\1) \hookrightarrow m \xtwoheadrightarrow{\sim} |\F|(\1).$$
After this we want to set $\Sim(\F)(\1)=m$ and $|\Sim(\F) |= |\F|$. This gives a $2$-constant premonoid that satisfies the co-Segal conditions. The purpose of the following discussion is to show that this can be done as $\kb_2$-injective replacement in $\coms$ where $\kb_2$ is a subset of the localizing set $\kb(\I)$.
\begin{df}
Define the \emph{minimal localizing set} for $2$-constant commutative premonoids as:
$$\kb_2= \{\Psi_\2(\av); \quad \alpha \in \I  \}.$$
\end{df}

\subsubsection{Pushout of an element of $\kb_2$}
Let $\Psi_\2(\alpha)$ be an element of $\kb_2$. We want to calculate the pushout of such morphism. But before doing this, let's recall some facts about the adjunction $\Ev_{u_\2}: \coms \leftrightarrows \M^{\Un}: \Psi_\2$. 
\begin{rmk}\label{rmk-important-deux-constant}
Let $\F$ be a  $2$-constant commutative premonoid and let $\G$ be any commutative premonoid.
\begin{enumerate}
\item A commutative square in $\coms$
\begin{equation}\label{diag-av-msx}
\xy
(0,18)*+{\Psi_\2(\alpha)}="W";
(0,0)*+{\Psi_\2(\Id_V) }="X";
(30,0)*+{\G}="Y";
(30,18)*+{\F}="E";
{\ar@{->}^-{}"X";"Y"};
{\ar@{->}^-{\Psi_\2(\av)}"W";"X"};
{\ar@{->}^-{\sigma}"W";"E"};
{\ar@{->}^-{\theta}"E";"Y"};
\endxy
\end{equation}
is equivalent by adjointness to a  commutative square in $\M^{\Un}$:
\begin{equation}\label{diag-av-arrow}
\xy
(0,18)*+{\alpha}="W";
(0,0)*+{\Id_V }="X";
(30,0)*+{\G(u_\2)}="Y";
(30,18)*+{\F(u_\2)}="E";
{\ar@{->}^-{}"X";"Y"};
{\ar@{->}^-{\av}"W";"X"};
{\ar@{->}^-{\sigma}"W";"E"};
{\ar@{->}^-{\theta}"E";"Y"};
\endxy
\end{equation}

\item  The commutative diagram \eqref{diag-av-arrow} in $\M^{\Un}$ is equivalent to a commutative cube in $\M$. And if we write this cube we find that  \eqref{diag-av-arrow} is equivalent to having the commutative diagram \eqref{diag-av-m} below, together with a lifting for the square defined by $\alpha$ and $\G(u_\2)$.
\begin{equation}\label{diag-av-m}
\xy
(0,20)*+{U }="A";
(25,20)*+{\F(\1)}="B";
(0,0)*+{V}="C";
(25,0)*+{|\F|(\1)}="D";
(50,20)*+{\G(\1)}="E";
(50,0)*+{\G(\2)}="F";
{\ar@{->}^-{}"A";"B"};
{\ar@{->}_-{\alpha}"A"+(0,-3);"C"};
{\ar@{->}_-{}"C"+(3,0);"D"};
{\ar@{->}^-{}_{}"B"+(0,-3);"D"};
{\ar@{->}^-{}"B";"E"};
{\ar@{->}_-{}^-{}"D";"F"};
{\ar@{->}^-{}"E";"F"};
{\ar@{.>}^-{}"C";"E"};
\endxy
\end{equation}

In this diagram, the square on the left represents $\sigma$, and the one on the right represents $\theta$. The whole commutative square represents the composite $\theta \circ \sigma$. The lifting $V \to \G(\1)$ and the commutativity of the lower triangle determine the map $\Id_V \to \G(u_\2)$ in the diagram \eqref{diag-av-arrow}.

\end{enumerate}
\end{rmk}
\begin{lem}\label{lem-deux-fondamental}
Let $\F$ be a  $2$-constant commutative premonoid such that $\F_{\geq 2}$ is a strict commutative monoid. Let $\varepsilon: \F \to \Ea$ be the pushout of $\Psi_\2(\av)$  along an attaching map $\sigma : \Psi_\2(\av) \to \F$:
\[
\xy
(0,18)*+{\Psi_\2(\alpha)}="W";
(0,0)*+{\Psi_\2(\Id_V) }="X";
(40,0)*+{\Ea}="Y";
(40,18)*+{\F}="E";
{\ar@{->}^-{}"X";"Y"};
{\ar@{->}^-{\Psi_\2(\av)}"W";"X"};
{\ar@{->}^-{\sigma}"W";"E"};
{\ar@{->}^-{\varepsilon}"E";"Y"};
\endxy
\]

Then $\Ea$ is also a \ul{$2$-constant} commutative premonoid such that $\Ea_{\geq \2}$ is strict commutative monoid. Moreover the following hold.
\begin{enumerate}
\item The natural transformation $\varepsilon_{\geq 2}:  \F_{\geq \2} \to \Ea_{\geq \2}$ is an isomorphism in $\Hom(\Phepiop_{\geq \2}, \M)$.
\item In particular we have an isomorphism $\F_{\geq \2} \cong \Ea_{\geq \2}$ of  strict commutative monoids
\end{enumerate}
\end{lem}

\begin{proof}
We will construct the $2$-constant diagram $\Ea$ and show that it satisfies the universal property of the pushout. For simplicity we will denote by $|\F|(\1)= \F_{\geq 2}(\1)= \F_{\geq 2}(\n)$. This is the underlying object of the strict commutative monoid  defined by $\F_{\geq 2}$. \\ 

First observe that since $\F$ is $2$-constant, the map $$\F(u_\2):\F(\1) \to \underbrace{\F(\2)}_{=|\F|(\1)}$$ is simultaneously part of $\F$ and is the canonical map $\F \to \F_{\geq \2}$ that connects $\F$ and $\F_{\geq \2}$.\\

 Technically  it's not hard to see why the lemma holds. Indeed we know that the $|\Psi_\2(\av)|$ is an isomorphism and the functor $|-|: \coms \to \com$ is a left adjoint so it preserves any kind of colimits. In particular it preserve pushouts. 

It follows that if we project the pushout square in $\com$ through $|-|$, we get also a pushout square there. In particular  the morphism $|\varepsilon|:|\F| \to |\Ea|$ is the pushout of the isomorphism  $|\Psi_\2(\av)|$, therefore it's an isomorphism. Now we have an isomorphism $|\F| \cong \F_{\geq \2}$ and we shall see that we also have an isomorphism $|\Ea| \cong \Ea_{\geq \2}$.
\begin{claim}
The idea of the proof is that taking the pushout is equivalent to factoring this map through a pushout of $\alpha$.\\
\end{claim}

Consider the attaching map  $\sigma: \Psi_\2(\alpha) \to \F$. By adjointness, this map corresponds to a unique morphism in $\M^{\Un}$, which in turn corresponds to a commutative square in $\M$: 

\begin{equation}\label{attach-push-e-av}
\xy
(0,18)*+{U}="W";
(0,0)*+{V}="X";
(30,0)*+{|\F|(\1)}="Y";
(30,18)*+{\F(\1)}="E";
{\ar@{->}^-{p}"X";"Y"};
{\ar@{->}^-{\alpha}"W";"X"};
{\ar@{->}^-{q}"W";"E"};
{\ar@{->}^-{\F(u_\2)}"E";"Y"};
\endxy
\end{equation}
Define $\Ea(\1)$ as the object we get when we take the pushout of $\alpha$ along $q$:
\begin{equation}\label{push-alpha-q}
V \xleftarrow{\alpha} U \xrightarrow{q} \F(\1).
\end{equation}
Let $\varepsilon:\F(\1) \to \Ea(\1)$ and $i_V: V \to \Ea(\1)$ be the canonical maps. The map $\varepsilon:\F(\1) \to \Ea(\1)$ is by definition the pushout of the map $\alpha$ along $q$.\\

If we use the universal property of the pushout with respect to the commutative square \eqref{attach-push-e-av} above, we find a unique map  $\gamma: \Ea(\1) \to |\F|(\1)$ such that the factorizations hereafter hold.
\begin{equation}\label{fact-f-us}
\F(u_\2)=\gamma \circ \varepsilon;
\end{equation}
\begin{equation}\label{fact-p}
 p= \gamma \circ i_V.
\end{equation}

Let's denote by $h:\Ea(\1) \to |\F|(\1)$ be the universal map $\gamma$ above.\\

Define $\Ea \in \coms$ to be the commutative premonoid obtained with the construction described in Scholium \ref{schol1} applied to the usual strict commutative monoid $\F_{\geq 2}$ with respect to the map $h$.\\

We have the following characteristics.
\begin{itemize}[label=$-$]
\item We have $ \Ea_{\geq 2}= \F_{\geq 2},$.
\item The map  $\F_{\geq 2} \to \Ea_{\geq 2}$ is the identity.
\item By definition the map $\Ea(u_\2)$ is just $\gamma=h$
\item The structure map $\Ea(u_\2)$ is simultaneously part of $\Ea$ and is the component of the canonical map that comes with the construction in Scholium \ref{schol1}:
$$\underbrace{h^{\star}\F_{\geq \2}}_{=\Ea} \to \F_{\geq \2}.$$
\item Finally, it's important to notice that the identity map $\F_{\geq 2} \to \Ea_{\geq 2}$ extends to a map $\Upsilon: \F \to \Ea$, whose component at the entry $\1$ is  the map $\varepsilon: \F(\1) \to \Ea(\1)$, which is the pushout of $\alpha$ along $q$. The component of $\Upsilon$ at every other entry is the identity map.
\end{itemize} 

If we incorporate the pushout object $\Ea(\1)$ in the diagram   \eqref{attach-push-e-av}, we find that the map $\Upsilon: \F \to \Ea$ fits in the following commutative diagram.
\begin{equation}\label{expand-univ-push-av}
\xy
(0,22)*+{U}="W";
(0,0)*+{V}="X";
(35,0)*+{|\F|(\1)}="Y";
(35,22)*+{\F(\1)}="E";
(15,10)*+{\Ea(\1)}="F";
{\ar@{->}_-{p}"X";"Y"};
{\ar@{->}_-{\alpha}"W";"X"};
{\ar@{->}^-{q}"W";"E"};
{\ar@{->}^-{\F(u_\2)}"E";"Y"};
{\ar@{.>}^-{\varepsilon}"E";"F"};
{\ar@{.>}^-{i_V}"X";"F"};
{\ar@{.>}^-{\gamma}"F";"Y"};
\endxy
\Longleftrightarrow
\xy
(0,22)*+{U }="A";
(25,22)*+{\F(\1)}="B";
(0,0)*+{V}="C";
(25,0)*+{|\F|(\1)}="D";
(60,22)*+{\Ea(\1)}="E";
(60,0)*+{|\F|(\1)=\Ea(\2)}="F";
{\ar@{->}^-{q}"A";"B"};
{\ar@{->}_-{\alpha}"A"+(0,-3);"C"};
{\ar@{->}_-{p}"C"+(3,0);"D"};
{\ar@{->}^-{}_{}"B"+(0,-3);"D"};
{\ar@{->}^-{\Upsilon=\varepsilon}"B";"E"};
{\ar@{->}_-{\Id}^-{\Upsilon}"D";"F"};
{\ar@{->}^-{\gamma=\Ea(u_\2)}"E";"F"};
{\ar@{.>}^-{}"C";"E"};
\endxy
\end{equation}
 
Now if we look at the right hand side of this diagram, then following Remark \ref{rmk-important-deux-constant}, we get by adjointness, a commutative square in $\coms$:
\[
\xy
(0,18)*+{\Psi_\2(\alpha)}="W";
(0,0)*+{\Psi_\2(\Id_V) }="X";
(40,0)*+{\Ea}="Y";
(40,18)*+{\F}="E";
{\ar@{->}^-{}"X";"Y"};
{\ar@{->}^-{\Psi_\2(\av)}"W";"X"};
{\ar@{->}^-{\sigma}"W";"E"};
{\ar@{->}^-{\Upsilon}"E";"Y"};
\endxy
\]

The rest of the proof is to show that this is the universal commutative square i.e, that $\Ea$ equipped with the maps above satisfies the universal property of the pushout. \\

Let $\G$ be an arbitrary commutative premonoid equipped with a copushout data $ \Psi_\2(\Id_V) \xrightarrow{\chi} \G \xleftarrow{\theta} \F$ that completes $$ \Psi_\2(\Id_V) \xleftarrow{\Psi_\2(\av)} \Psi_\2(\alpha) \xrightarrow{\sigma} \F$$
into a commutative square. Then following Remark \ref{rmk-important-deux-constant}, having such commutative square is uniquely equivalent to having the commutative square below and  a lifting. 
\[
\xy
(0,20)*+{U }="A";
(25,20)*+{\F(\1)}="B";
(0,0)*+{V}="C";
(25,0)*+{|\F|(\1)}="D";
(50,20)*+{\G(\1)}="E";
(50,0)*+{\G(\2)}="F";
{\ar@{->}^-{q}"A";"B"};
{\ar@{->}_-{\alpha}"A"+(0,-3);"C"};
{\ar@{->}_-{p}"C"+(3,0);"D"};
{\ar@{->}^-{}_{}"B"+(0,-3);"D"};
{\ar@{->}^-{}"B";"E"};
{\ar@{->}_-{}^-{}"D";"F"};
{\ar@{->}^-{}"E";"F"};
{\ar@{.>}^-{}"C";"E"};
\endxy
\] 

Since the upper half triangle that ends at $\G(\1)$ is commutative, the universal property of the pushout of $\alpha$ along $q$, gives a \emph{unique map} $\zeta: \Ea(\1) \to \G(\1)$ that satisfies the usual factorizations. 

Another application of the universal property of the pushout with respect to the whole commutative square that ends at  $\G(\2)$, gives a unique map $\rho: \Ea(\1) \to \G(\2)$ that satisfies the usual factorizations. But the maps $\theta \circ \gamma$ and $\G(u_\2) \circ \zeta$ both solve for $\rho$ the required factorizations, therefore by uniqueness of $\rho$ we have an equality:
$$ \theta \circ \gamma= \G(u_\2) \circ \zeta.$$ 

The above facts can be summarized by saying that everything commutes in the  following diagram.
\begin{equation}\label{whole-diag-com}
\xy
(0,20)*+{\F(\1) }="A";
(25,20)*+{\Ea(\1)}="B";
(0,0)*+{|\F|(\1)}="C";
(25,0)*+{|\F|(\1)}="D";
(50,20)*+{\G(\1)}="E";
(50,0)*+{\G(\2)}="F";
{\ar@{->}^-{}"A";"B"};
{\ar@{->}_-{\F(u_\2)}"A"+(0,-3);"C"};
{\ar@{->}_-{\Id}"C";"D"};
{\ar@{->}^-{\Ea(u_\2)}_{}"B"+(0,-3);"D"};
{\ar@{->}^-{\zeta}"B";"E"};
{\ar@{->}_-{}^-{}"D";"F"};
{\ar@{->}^-{}"E";"F"};
(-25,20)*+{U}="U";
(-25,0)*+{V}="V";
{\ar@{->}_-{\alpha}"U"+(0,-3);"V"};
{\ar@{->}_-{p}"V";"C"};
{\ar@{->}^-{q}"U";"A"};
{\ar@{.>}_-{}"V";"E"};
{\ar@{.>}_-{}"V";"B"};
\endxy
\end{equation}

Since $\Ea=\F$ everywhere except at the entry $(\1)$, we see that  the map $\zeta$ and the data for the map $\theta : \F \to \G$ determine  a unique map that we denote again by $ \zeta: \Ea \to \G$. This map gives the factorization:
\begin{equation}\label{first-fact-upsilon}
\theta= \zeta \circ \Upsilon.
\end{equation}

Note also that the lifting $V \to \G(\1)$ is precisely the composite $\zeta \circ i_V$, where $i_V: V \to \Ea(\1)$ is the lifting for the square associated to $\Ea$. It follows that if we look at the diagram \eqref{whole-diag-com} in the arrow category $\M^{\Un}$, we get a commutative diagram:

\begin{equation}\label{diag-av-arrow-proof}
\xy
(0,22)*+{\alpha}="W";
(0,0)*+{\Id_V }="X";
(40,0)*+{\G(u_\2)}="Y";
(40,22)*+{\F(u_\2)}="E";
(20,10)*+{\Ea(u_\2)}="F";
{\ar@{->}^-{}"X";"Y"};
{\ar@{->}^-{\av}"W";"X"};
{\ar@{->}^-{\sigma}"W";"E"};
{\ar@{->}^-{\theta}"E";"Y"};
{\ar@{.>}^-{\Upsilon}"E";"F"};
{\ar@{.>}^-{}"X";"F"};
{\ar@{.>}^-{\zeta}"F";"Y"};
\endxy
\end{equation}

Now the uniqueness of the adjunct map implies that the given map $\Psi_\2(\Id_V) \xrightarrow{\chi} \G$ is the composite of the canonical map $\Psi_\2(\Id_V) \to \Ea$ and $\zeta$. Putting this together with the previous factorization \eqref{first-fact-upsilon}, we see that $\Ea$ satisfies the universal property of the pushout and the lemma follows.
\end{proof}

The next lemma tells us how to calculate the wide pushout of maps  $\F \to \Ea$ such as the one we've just constructed. 
\begin{lem}\label{lem-wide pushouts-deux-constants}
Let $\{\Upsilon_i: \F \to \Ea_i \}_{i \in S}$ be a small family of morphisms between $2$-constant commutative premonoids in $\coms$. Assume that each morphism 
$\Upsilon_i: \F \to \Ea_i$ is such that the induced morphism $\Upsilon_{i,\geq 2}: \F_{\geq 2} \to \Ea_{i,\geq 2} $ is an isomorphism of usual commutative monoids.\\

Let $\Ea_{\infty}$ be the wide pushout in $\coms$ of the maps $\Upsilon_i$. Then the following hold.
\begin{enumerate}
\item $\Ea_\infty$ is also a $2$-constant commutative premonoid.
\item Each canonical map  $\Ea_i \to \Ea_\infty$  and $ \F \to \Ea_\infty$ induces an isomorphism of strict commutative monoids.  
$$\Ea_{i,\geq 2}\xrightarrow{\cong} \Ea_{\infty,\geq 2}, \quad \F_{\geq 2}  \xrightarrow{\cong} \Ea_{\infty,\geq 2}$$
\end{enumerate}
 \end{lem}
\begin{proof}
Take $\Ea_{\infty,\geq 2}$ to be the object obtained by taking the wide pushout of the isomorphisms $\Upsilon_{i,\geq 2}: \F_{\geq 2} \to \Ea_{i,\geq 2}$. Clearly the canonical maps $ \Ea_{i,\geq 2}\xrightarrow{\cong} \Ea_{\infty,\geq 2}$ and $\F_{\geq 2} \xrightarrow{\cong} \Ea_{\infty,\geq 2}$ are isomorphisms. 

One gets the initial entry 
$\Ea_{\infty}(\1)$ together with the unique map $\Ea_{\infty}(\1) \xrightarrow{} \Ea_{\infty,\geq 2}(\1)$ by taking the wide pushout of $\F(u_\2) \to \Ea_i(u_\2)$ in the arrow category $\M^{\Un}$. Just like before $\Ea_\infty$ is also a commutative premonoid. The canonical maps $\Ea_{i,\geq 2}\xrightarrow{\cong} \Ea_{\infty,\geq 2}$ and $\F_{\geq 2} \xrightarrow{\cong} \Ea_{\infty,\geq 2}$ extend to morphisms in $\coms$: 
$$\Ea_i \to \Ea_\infty ; \quad \F \to \Ea_\infty.$$

Moreover for each $i$, the canonical  map  $\F \to \Ea_\infty$ is the composite of  $\F \to \Ea_i$ and $\Ea_i \to \Ea_\infty$. This means that we have a natural cocone that ends at $\Ea_\infty$. The reader can easily check that this cocone is the universal one i.e, $\Ea_\infty$ equipped with this cocone satisfies the universal property of the wide pushout.
\end{proof}
\subsubsection{Analysis of the $\kb_2$-injective replacement functor}
\begin{pdef}\label{pdef-deux-constant}
Let $\Sim_2: \coms \to \coms$ be the $\kb_2$-injective replacement functor obtained by the small object argument. Denote by $\tau: Id \to \Sim$ the induced natural transformation.\\

Let  $\F$ be a  $2$-constant commutative premonoid such that $\F_{\geq 2}$ is a strict  usual commutative  monoid. Then the following hold.

\begin{enumerate}
\item $\Sim(\F)$ is also a  $2$-constant commutative premonoid such that $\Sim(\F)_{\geq 2}$ is  usual commutative  monoid.  The map $\tau_{\geq 2}: \F_{\geq 2} \to \Sim(\F)_{\geq 2}$ is an isomorphism of usual commutative  monoids.
\item $\Sim(\F)$ satisfies the co-Segal conditions. And the canonical map $\Sim(\F) \to |\Sim(\F)|$ is an easy weak equivalence. Moreover we have an isomorphism $|\Sim(\F)| \cong \Sim(\F)_{\geq 2}$.
\item Let $L: \coms\to  \Ba$ be a functor that sends easy weak equivalences to isomorphisms and takes any $\kb_2$-cell complex to an isomorphism. Then for all $\F \in \coms$ not necessarily $2$-constant, the image by $L$ of the unit $\eta: \F \to \iota(|\F|)$ is an isomorphism in $\Ba$. 
\end{enumerate}
The functor $\Sim$ will be called the \emph{ $2$-constant co-Segalification functor}. 
\end{pdef}

\begin{proof}
The full subcategory of $2$-constant precategories is closed under directed colimits (and limits) since they are computed level-wise. Thanks to Lemma \ref{lem-deux-fondamental}, we know that if $\F$ is $2$-constant, then the pushout of any $\Psi_\2(\av)$ along any  $\Psi_\2(\alpha) \to \F$ is a morphism of $2$-constant commutative premonoids which is moreover an isomorphism on the \emph{underlying categories}.\\

Given any pushout data defined by a coproduct of maps in $\kb_2$
$$(\coprod \Psi_\2(\Id_V)) \xleftarrow{\coprod \Psi_\2(\av)} (\coprod \Psi_\2(\alpha)) \to \F, $$
we know thanks to  Lemma \ref{lem-wide pushouts-deux-constants} and Lemma \ref{coproduc-pushout}  that the canonical map $\F \to \F_1$ toward the pushout-object is again a map of $2$-constant commutative premonoids. Moreover the induced map $|\F| \to |\F_1|$ is an isomorphism of  usual commutative  monoids.  But these pushouts are precisely the ones we use to construct $\Sim(\F)$.\\ 

Therefore  $\Sim(\F)$ is a $2$-constant commutative premonoid as a directed colimit of $2$-constant commutative premonoids;  and the map 
$ \eta: \F \to \Sim(\F)$ is a  $\kb_2$-cell complex with the property that the induced map $|\eta|: |\F| \to |\Sim(\F)|$ is an isomorphism of  usual commutative  monoids. This proves Assertion $(1)$. \\

Assertion $(2)$ follows from the fact that $\Sim(\F)$ in $\kb_2$-injective, and by adjointness this means that the unique map $\Sim (\F)(\1) \to |\Sim(\F)|(\1)$, viewed as an object of $\M^{\Un}$, is $\av$-injective for all generating cofibration $\alpha$.\\
But this in turn simply means that we have a lifting to any problem defined by $\alpha$ and  $\Sim (\F)(\1) \to |\Sim(\F)|(\1)$ (see Proposition \ref{lem-lifting-inject}). Consequently $\Sim(\1) \to |\Sim|(\1)$ is a trivial fibration, in particular a weak equivalence, therefore $\Sim(\F)$ is a co-Segal commutative monoid. This also proves at the same time that the canonical map $\Sim(\F) \to |\Sim(\F)|$, whose component at $\1$ is exactly $\Sim (\F)(\1) \to |\Sim(\F)|(\1)$ is an easy weak equivalence.\\
 
Now for Assertion $(3)$  it suffices to use the factorization $\F \to \iota(|\F|)$ given in Definition \ref{rmk-associated-deux-constant} and observe that we can factor  again $\F \to \iota(|\F|)$ as follows.

$$\F \xrightarrow{\rho} h^\star \iota(|\F|)  \xrightarrow{\tau} \Sim[h^\star \iota(|\F|)] \xrightarrow{\sim}  \Sim[h^\star \iota(|\F|)]_{\geq 2} \xrightarrow{\cong} \iota(|\F|).$$

Since $L$ sends every $\kb_2$-cell complex to an isomorphism,  then $L$ sends  the map $\tau$  to an isomorphism in $\Ba$. The map  $\Sim[h^\star \iota(|\F|)] \xrightarrow{\sim}  \Sim[h^\star \iota(|\F|)]_{\geq 2}$ is an easy weak equivalence by the previous assertion we've just proved. The map $\rho$ is always an easy weak equivalence thanks to Proposition \ref{prop-equiv-constant}. In the end we find that the image of  $\F \to |\F|$ by $L$ is also an isomorphism.
\end{proof}

\section{New model structure on commutative premonoids}\label{section-new-model-str}

\subsection{Enlarging the cofibrations}

The discussion that follows is motivated by the desire to have a combinatorial left proper model structure on $\coms$ such that the set of generating cofibrations contains the localizing set $\kb(\I)$ introduced above.
\begin{nota}
\begin{enumerate}
\item Denote by $\I_{\coms}$ the generating set for the cofibrations in the model category $\comse$ of Theorem \ref{easy-model-coms}.
\item Denote by $\I_{\coms}^{+}$ the set:
$$\I_{\coms} \bigsqcup \kb(\I).$$
\end{enumerate}
\end{nota}

The following result is just an observation.

\begin{lem}\label{petit-lem-cofib}
For all $\sigma \in \I_{\coms}^{+}$ the component  $\sigma_{\1}$ is a cofibration in $\M$
\end{lem}
\begin{proof}
The statement is clear if $\sigma \in \I_{\coms}$. If $\sigma= \Psi_\n(\av) \in \kb(\I)$,with $\alpha \in \I$,  the component $\sigma_{\1}$ is $\Id_I \coprod \alpha$ (see Proposition \ref{prop-comp-kbi}).
\end{proof}
\subsection{The model structure}
We show below that there is a left proper combinatorial model structure on $\coms$ with $\I_{\coms}^{+}$  as the set of generating cofibrations and $\W_{\comse}$ as the class of weak equivalences.\\

We use Smith's recognition Theorem for combinatorial model categories (see for example Barwick \cite[Proposition 2.2]{Barwick_localization}). This theorem gives the possibility to  construct a combinatorial model category out of two data consisting of a class $\W$ of morphisms whose elements are called \emph{weak equivalences}; and a set $\I$ of \emph{generating cofibrations}.\\

Our method is classical and the argument is present in Pellissier's PhD thesis \cite{Pel}; it is also used by Lurie \cite{Lurie_HTT}, Simpson \cite{Simpson_HTHC} and others. But in doing so, we actually reprove (implicitly) a derived version of Smith's theorem that has been outlined by Lurie \cite[Proposition A.2.6.13]{Lurie_HTT}. This version asserts that the resulting combinatorial model structure is automatically left proper. So we will just use that proposition  that we recall hereafter with the same notation as in Lurie's book. 
\begin{prop}\label{Smith-Lurie}
Let $\bf{A}$ be a presentable category. Suppose we are given a class $W$ of morphisms of A, which we will call weak equivalences, and a (small) set $C_0$ of morphisms of $\bf{A}$, which we will call generating cofibrations. Suppose furthermore that the following assumptions are satisfied:
\begin{itemize}
\item[$(1)$] The class $W$ of weak equivalences is perfect (\cite[Definition A.2.6.10]{Lurie_HTT}). 
\item[$(2)$] For any diagram
\[
\xy
(0,15)*+{X}="A";
(20,15)*+{Y}="B";
(0,0)*+{X'}="C";
(20,0)*+{Y'}="D";
(0,-15)*+{X''}="X";
(20,-15)*+{Y''}="Y";
{\ar@{->}^-{f}"A";"B"};
{\ar@{->}_-{}"A";"C"};
{\ar@{->}^-{}"B";"D"};
{\ar@{->}^-{}"C";"D"};
{\ar@{->}^-{}"X";"Y"};
{\ar@{->}^-{g}"C";"X"};
{\ar@{->}^-{g'}"D";"Y"};
\endxy
\] 

in which both squares are coCartesian (=pushout square), $f$ belongs to $C_0$, and $g$ belongs $W$, the map $g'$ also below to $W$.
\item[$(3)$] If $g: X \to Y$ is a morphism in $\bf{A}$ which has the right lifting property with respect to every morphism in $C_0$, then $g$ belongs to $W$.
\end{itemize}

Then there exists a left proper combinatorial model structure on $\bf{A}$ which
may be described as follows:
\begin{itemize}
\item[$(C)$] A morphism $f : X \to  Y$ in $\bf{A}$ is a cofibration if it belongs to the weakly
saturated class of morphisms generated by $C_0$.
\item[$(W)$]  A morphism $f:X\to Y$ in $\bf{A}$ is a weak equivalence if it belongs to $W$.
\item[$(F)$] A morphism $f:X\to Y$ in $\bf{A}$ is a  fibration if it has the right lifting property with respect to every map which is both a cofibration and a weak equivalence.
\end{itemize}
\end{prop}
\begin{note}
Here \emph{perfectness} is a property of stability under filtered colimits and a generation by a small set $W_0$ (which is more often the intersection of $W$ and the set of maps between presentable objects). The reader can find the exact definition in \cite[Definition A.2.6.10]{Lurie_HTT}.
\end{note}

\begin{warn}
We've used so far the letters $f, g$ as functions so to avoid any confusion we will use $\sigma,\sigma'$ instead.
\end{warn}

Applying the previous proposition we get the following theorem.
\begin{thm}\label{enlarging-msx}
Let $\M$ be a combinatorial monoidal model category which is left proper. Then  there exists a combinatorial model structure on $\coms$ which is left proper and which may be described as follows. 
\begin{enumerate}
\item A map $\sigma: \F \to \G$ is a weak equivalence if it's an easy weak equivalence i.e, if it's in $\W_{\comse}$.
\item A map $\sigma: \F \to \G$  is a cofibration if it belongs to the weakly
saturated class of morphisms generated by $ \I_{\coms}^{+}$.
\item A morphism $\sigma: \F \to \G$ is a fibration if it has the right lifting property with respect to every map which is both a cofibration and a weak equivalence
\end{enumerate}
We will denote this model category  by $\comsep$. The identity functor $$ \Id: \comse \to \comsep, $$ is a left Quillen equivalence
\end{thm}
\begin{proof}
Condition $(1)$ is straightforward because $\W_{\comse}$ is the class of weak equivalences in the combinatorial model category $\comse$. We also have Condition $(3)$ since a map $\sigma$ in $\I_{\coms}^{+}\tx{-inj}$ is in particular in $\I_{\coms}\tx{-inj}$, therefore it's a trivial fibration in $\comse$ and thus an easy weak equivalence.\\

It remains to check that Condition $(2)$ is also satisfied. Consider the following diagram as in the proposition.
\[
\xy
(0,10)*+{\F}="A";
(20,10)*+{\G}="B";
(0,0)*+{\F'}="C";
(20,0)*+{\G'}="D";
(0,-10)*+{\F''}="X";
(20,-10)*+{\G''}="Y";
{\ar@{->}^-{\sigma}"A";"B"};
{\ar@{->}_-{}"A";"C"};
{\ar@{->}^-{}"B";"D"};
{\ar@{->}^-{}"C";"D"};
{\ar@{->}^-{}"X";"Y"};
{\ar@{->}^-{\theta}"C";"X"};
{\ar@{->}^-{\theta'}"D";"Y"};
\endxy
\]

 If $\sigma: \F \to \G$ is in $\I_{\coms}^{+}$, we have from Lemma \ref{petit-lem-cofib} that each top-component 
$$\sigma_{\1}: \F(\1) \to \G(\1),$$ 
is a cofibration in $\M$. Now as mentioned several times in the paper, pushouts in $\coms$ are computed level-wise at the entry $\1$. It follows that the top components in that diagram are obtained by pushout in $\M$; and since $\M$ is left proper we get that the component $\theta'_{\1}$ is a weak equivalence, which means that $\theta'$ is an easy weak equivalence as desired.\\

A cofibration in $\comse$ is a cofibration in $\comsep$ and since we have the same weak equivalences  then the identity functor is a Quillen equivalence.
\end{proof}
\section{Bousfield localizations}\label{sec-bousfield}
\begin{warn}
We would like to warn the reader about our upcoming notation for the left Bousfield localizations. We choose to include a small letter $\mathbf{c}$ (for ``correct'') as a superscript in both $\comsep$ and $\comse$ to mean that we are taking the left Bousfield localization with respect to the (same) set $\kb(\I)$. A more suggestive and standard notation should be $\Lb^{\kb(\I)}\comse$ or $\kb(\I)^{-1}\comse$, but as the reader can see, this is too heavy to work with.\\

Instead we will use the notation $\comsec$ and $\comsepc$.
\end{warn}

\begin{rmk}

\begin{enumerate}
\item Since we have the same class of weak equivalences in $\comse$ and in $\comsep$, we have an equivalence of function complexes on these model structures. It follows that a map $\sigma$ in $\coms$ is a $\kb(\I)$-local equivalence in the model structure $\comse$ if and only if it's a $\kb(\I)$-local equivalence in the model structure $\comsep$.
\item A direct consequence of this is that the left Quillen equivalence $$\comse \to \comsep,$$ given by the identity, will pass to a left Quillen equivalence between the respective Bousfield localization:
$$\comsec \to \comsepc.$$ 
\end{enumerate}
\end{rmk}

\subsection{The first localized model category}
We start with the localization of $\comsep$.
\begin{thm}\label{main-thm-1}
Let $\M$ be a combinatorial monoidal model category which is left proper. Then  there exists a combinatorial model structure on $\coms$ which is left proper and which may be described as follows. 
\begin{enumerate}
\item A map $\sigma: \F \to \G$ is a weak equivalence if and only if it's a $\kb(\I)$-local equivalence.
\item A map $\sigma: \F \to \G$ is cofibration if it's a cofibration in $\comsep$.
\item Any fibrant object $\F$ is a co-Segal commutative monoid. 
\item We will denote this model category  by $\comsepc$. 
\item The identity of $\coms$ determines the universal left Quillen functor
$$\Lb_+: \comsep \to \comsepc.$$

\end{enumerate}

This model structure is the left Bousfield localization of $\comsep$ with the respect to the set $\kb(\I)$. \end{thm}
\begin{df}
Define a \emph{co-Segalification functor} for $\coms$ to be any fibrant replacement functor in $\comsepc$.
\end{df}

\begin{proof}[Proof of Theorem \ref{main-thm-1}]
The existence of the left Bousfield localization and the left properness is guaranteed by Smith's theorem on left Bousfield localization for combinatorial model categories. We refer the reader to Barwick \cite[Theorem 4.7]{Barwick_localization} for a precise statement. This model structure is again combinatorial.\\

For the rest of the proof we will use the following facts on Bousfield localization and the reader can find them in Hirschhorn's book \cite{Hirsch-model-loc}. 
\begin{enumerate}
\item A weak equivalence in $\comsepc$ is a \emph{$\kb(\I)$-local weak equivalence}; we will refer them as \emph{new weak equivalence}. And any easy weak equivalence (old one) is a new weak equivalence.
\item The new cofibrations are the same as the old ones and therefore the new trivial fibrations are just the old ones too. In particular a trivial fibration in the left Bousfield localization is an easy weak equivalence. 
\item The fibrant objects are the  $\kb(\I)$-local objects that are fibrant in the original model structure.
\item Every map in $\kb(\I)$ becomes a weak equivalence in $\comsec$, therefore an isomorphism in the homotopy category.
\end{enumerate}

Let $\F$ be a fibrant object in $\comsepc$, this means that the unique map $\F \to \ast$ has the RLP with respect to any trivial cofibration. Now observe elements of $\kb(\I)$ are trivial cofibrations in $\comsepc$ because they were old cofibrations and become weak equivalences. So any fibrant $\F$ must be in particular $\kb(\I)$-injective and thanks to  Lemma \ref{k-inj-cosegal}, we know that $\F$ satisfies the co-Segal conditions.
\end{proof}
\begin{cor}
Let $\Sim $ be a co-Segalification functor, i.e a fibrant replacement in $\comsepc$. Then a map $\sigma: \F \to \G$ is a weak equivalence in $\comsepc$ if and only if the map 
$$\Sim(\sigma): \Sim(\F) \to \Sim(\G),$$
 is a level-wise weak equivalence of co-Segal commutative monoids.
\end{cor}

\begin{proof}
Since $\Sim$ is a fibrant replacement functor in the new model structure, then $\Sim(\F)$ is a co-Segal commutative monoid for all $\F$,  by the second assertion of the previous theorem.

 By the $3$-for-$2$ property of weak equivalences in any model category, a map $\sigma$ is a weak equivalence if and only if $\Sim(\sigma)$ is a weak equivalence. But $\Sim(\sigma)$ is a weak equivalence of fibrant objects in the Bousfield localization, therefore it's a weak equivalence in the original model structure.

In the end we see that  $\sigma$ is a weak equivalence in $\comsepc$ if and only if $\Sim(\sigma)$ is an easy weak equivalence in $\comsep$. Now an easy weak equivalence between co-Segal categories is just a level-wise weak equivalence.
\end{proof}
\begin{prop}\label{prop-eta-kx-loc-equiv}
For any $\F \in \coms$, the canonical map 
$\F \to |\F|$ is an equivalence in $\comsepc$ i.e, it's a $\kb(\I)$-local equivalence in $\comsep$ (whence in $\comse$).  
\end{prop}
\begin{proof}
Every element $\Psi_s(\av) \in \kb(\I)$ becomes a trivial cofibration  in $\comsepc$, since they were cofibration in $\comsep$. In particular every $\kb(\I)$-cell complex (whence $\kb_2$-cell complex) is a trivial cofibration.  The proposition follows from  Assertion $(3)$ of Proposition \ref{pdef-deux-constant}.
\end{proof}
\subsection{The second localized model category}
We now localize the original model category $\comse$ that doesn't contain a priori the set $\kb(\I)$ among the class of cofibrations. The theorem we give below is also a straightforward application of Smith's theorem for left proper combinatorial model category.
\begin{thm}\label{main-thm-2}
Let $\M$ be a combinatorial monoidal model category which is left proper. Then  there exists a combinatorial model structure on $\coms$ which is left proper and which may be described as follows.
\begin{enumerate}
\item A map $\sigma: \F \to \G$ is a weak equivalence if and only if it's a $\kb(\I)$-local equivalence. 
\item A map $\sigma: \F \to \G$ is cofibration if it's a cofibration in $\comse$.
We will denote this model category  by $\comsec$
\item If a transferred model structure on $\com$ exists then the adjunction $$ |-|: \comse \leftrightarrows \com: \iota, $$ descends to a Quillen adjunction:
$$ |-|^{\mathbf{c}}: \comsec \leftrightarrows \com: \iota.$$
In particular the inclusion $\iota: \com \to \comsec$ is again a right Quillen functor.
\item  The identity of $\coms$ determines the universal left Quillen functor
$$\Lb: \comse \to \comsec.$$
\end{enumerate}
This model structure is the left Bousfield localization of $\comse$ with the respect to the set $\kb(\I)$. 
\end{thm}
\begin{proof}
The existence and characterization of the Bousfield localization follows also from Smith theorem. The first two assertions are clear. Assertion $(3)$ is a consequence of the universal property of the left Bousfield localization. Indeed we have a left Quillen functor $|-|: \comse \to \com$ that takes elements of $\kb(\I)$ to weak equivalences in $\com$ (Proposition \ref{prop-comp-kbi}). Therefore there exists a unique left Quillen functor $|-|^{\mathbf{c}}: \comsec \to \com$ such that we have an equality
$$|-|: \comse \to \com= \comse \xrightarrow{\Lb} \comsec \xrightarrow{|-|^{\mathbf{c}}} \com.$$
\end{proof}

\section{The Quillen equivalence}\label{sec-Quillen-equiv}

We start first with the following lemma which is useful to establish the Quillen equivalence. 
\begin{lem}\label{lem-reflect-equiv}
Let $\M$ be a symmetric monoidal model category that is combinatorial and left proper. Assume that the transferred (natural) model structure on $\com$ exists. Let $\sigma : \C \to \D$ be a morphism between usual commutative monoids regarded also as morphism in $\coms$. 

Then $\sigma$ is a $\kb(\I)$-local equivalence in $\comse$ (whence in $\comsep$) if and only if it's a weak equivalence in $\com$. 
\end{lem}
\begin{proof}
The if part is clear since a weak equivalence in $\com$ is an easy weak equivalence in $\coms$ and therefore it's also a weak equivalence in the Bousfield localization. But the weak equivalences in the Bousfield localization are precisely the $\kb(\I)$-local equivalences.\\

Let's now assume that $\sigma: \C \to \D$ becomes a $\kb(\I)$-local equivalence.\\

Use the axiom of factorization in the model category $\com$ to factor $\sigma$ as trivial cofibration followed by a fibration: 
$$\sigma= \C \xhookrightarrow[\sim]{\sigma_1} \Ea \xtwoheadrightarrow{\sigma_2} \D.$$  

Since we know that inclusion $\iota: \com \to \comse$ is a right Quillen functor when we pass to the left Bousfield localization $\comsec$, it follows that the map $\sigma_2: \Ea \to \D$ is a fibration in $\comsec$. Now  as $\sigma_1$ is a weak equivalence of usual commutative monoids, by the if part, it's also a $\kb(\I)$-local equivalence. \\

By the $3$-for-$2$ property of $\kb(\I)$-local equivalences applied to the equality $\sigma= \sigma_2 \circ \sigma_1$, we find that $\sigma_2$ is also a $\kb(\I)$-local equivalence. \\

In the end we see that $\sigma_2$ is simultaneously a fibration in the Bousfield localization $\comsec$ and a $\kb(\I)$-local equivalence, therefore it's a trivial fibration in $\comsec$. But a trivial fibration in this left Bousfield localization is the same as a trivial fibration in the original model structure. This means that $\sigma_2$ is usual trivial fibration of strict commutative monoids. 

Then $\sigma= \sigma_2 \circ \sigma$ is the composite of weak equivalence of usual commutative monoids and the lemma follows.
\end{proof}

\subsection{The main Theorem}
\begin{thm}\label{main-thm-quillen-equiv}
Let $\M$ be a symmetric monoidal model category that is combinatorial and left proper. Then the following hold. 
\begin{enumerate}
\item The left Quillen equivalence $\comse \to \comsep$ induced by the identity of $\coms$ descends to a left Quillen equivalence between the respective left Bousfield localizations with respect to $\kb(\I)$:
$$\comsec \to \comsepc.$$
\item If the transferred model structure on $\com$ exists, then the adjunction
$$ |-|^{\mathbf{c}}: \comsec \leftrightarrows \com: \iota,$$
is a Quillen equivalence. 
\item The diagram $\comsepc \xleftarrow{}\comsec \xrightarrow{|-|^{\mathbf{c}}} \com$ is a zigzag of Quillen equivalences. In particular we have a diagram of equivalences between the homotopy categories.
$$\Ho[\comsepc] \xleftarrow{\simeq} \Ho[\comsec] \xrightarrow{\simeq} \Ho[\com].$$
\end{enumerate}
\end{thm}
\begin{proof}
We only need to prove the second assertion, namely that we have a Quillen equivalence
$$ |-|^{\mathbf{c}}: \comsec \leftrightarrows \com: \iota.$$

For this it suffices to show that if $\F \in \comsec$ is cofibrant and if $\C\in \com$ is fibrant, then a map $\sigma:\F \to \iota(\C)$ is a weak equivalence in $\comsec$ if and only if the adjunct map $\ol{\sigma}:|\F| \to \C$ is a weak equivalence in $\com$.\\

Since the inclusion $\iota: \com \to \comsec$ is fully faithful, we will identity $\C$ with $\iota(\C)$ and $|\F|$ with $\iota(|\F|)$. Let $\eta: \F \to |\F|$ be the unit of the adjunction. Then all three maps fit in a commutative triangle in $\comsec$:
\begin{equation}\label{triang}
\xy
(0,15)*+{\F}="W";
(0,0)*+{|\F| }="X";
(20,0)*+{\C}="Y";
{\ar@{->}^-{\ol{\sigma}}"X";"Y"};
{\ar@{->}^-{\eta}"W";"X"};
{\ar@{->}^-{\sigma}"W";"Y"};
\endxy
\end{equation}

Thanks to Proposition \ref{prop-eta-kx-loc-equiv} we know that $\eta: \F \to |\F|$ is always a $\kb(\I)$-local equivalence. Then by $3$-for-$2$ we see that $\sigma$ is a $\kb(\I)$-local equivalence if and only if $\ol{\sigma}$ is. Now thanks to Lemma \ref{lem-reflect-equiv} we know that $\ol{\sigma}$ is a $\kb(\I)$-local equivalence if and only if it's an equivalence in $\com$ and the theorem follows.
\end{proof}

\newpage
\appendix
\section*{Proof of Lemma \ref{lem-reduc-adj-nass}}
For the reader's convenience we include a short discussion about the proof of Lemma \ref{lem-reduc-adj-nass}:
\begin{lemsans}
Let $\M$ be a symmetric monoidal closed category that is also locally presentable. Then there is a left adjoint to the forgetful functor 
$$\Lax(\Phepiop,\M)_{NA} \to \Hom(\Phepiopl,\M).$$

The induced adjunction is monadic and $\Lax(\Phepiop,\M)_{NA}$ is also locally presentable.
\end{lemsans}

\begin{nota}
We will use the following notation.
\begin{enumerate}
\item Let $\Add: \Phepiop \times \Phepiop \to \Phepiop$ be the bifunctor $+$ in the symmetric monoidal category $(\Phepiopl, +, \0)$.
\item Let's regard $\Add$ as a morphism in $\Cat$, therefore as a functor
$$\Add: \Un \to \Cat,$$ 
with $\Add(-)= \Phepiop \times \Phepiop$ and $Add(+)=\Phepiop$.
\item Let $\int \Add  \to \Un$ be the functor obtained from the Grothendieck construction applied to $\Add$. The fiber of $\ap$ is $\Phepiop \times \Phepiop $ and the fiber of $\bp$ is $\Phepiop$.
\item For $\n \in \Phepiop$, we will identify $\n$ with the object $(+,\n)$ in the total category $\int \Add$.
\item Denote by $[(\int Add)\downarrow \n]$ the usual comma category of the total category $\int \Add$ in which we've identified $\n$ with $(+,\n)$.
\item For simplicity we will denote by $\partial(Add \downarrow \n)\subset [(\int Add)\downarrow \n]$ the full subcategory obtained by removing the objects $[-,(\0,\n)]$, $[-,(\n,\0)]$ and $(+,\n)$ from the comma category
\item Let $\Hom([\Phepiopl]^{\leq \n},\M)$ be the category of $\n$-truncated diagrams.
\item Similarly let $\Lax([\Phepiop]^{\leq \n},\M)_{NA}$ be the corresponding category of truncated nonassociative data. We also have a forgetful functor 
$$\Ub^{\leq \n}:\Lax([\Phepiop]^{\leq \n},\M)_{NA} \to  \Hom([\Phepiopl]^{\leq \n},\M).$$
\end{enumerate}
\end{nota}
To prove our lemma we will construct the left adjoint $\Gamma$ by induction on the degree
\begin{schol}
Let $\n >0$ be an object of $\Phepiop$. Then any $(\n-\1)$-truncated object $\H \in \Lax([\Phepiop]^{\leq (\n-\1)},\M)_{NA}$ gives rise to a functor
$$ \partial_{\n}\H: \partial(Add \downarrow \n) \to \M,$$
defined by the formulas below.
\begin{enumerate}
\item For $u:[\ap,(\p,\q)] \to (\bp,\n)$ in $\partial(Add \downarrow \n)$ we define $\partial_{\n}\H(u)$ as follows.
\begin{itemize}[label=$-$]
\item If $\p \neq \0$ and $\q \neq \0$ we set
$$\partial_{\n}\H(u):= \H(p) \otimes \H(q).$$
\item If $\p=\0$ and $\q \neq \0$ we set
$$\partial_{\n}\H(u):= I \otimes \H(\q) \quad (\cong \H(\q)).$$
\item Similarly if $\p\neq \0$ and $\q = \0$ we set
$$\partial_{\n}\H(u):= \H(\p) \otimes I \quad (\cong \H(\p)).$$
\end{itemize}
\item And given $u$ of the form $u: (+,\p) \to (+,\n)$, we simply take $\partial_{\n}\H(u)= \H(\p).$
\item To define the image by $\partial_{\n}\H$ of a morphism in $\partial(Add \downarrow \n)$, it suffices to define the image of each canonical map $\gamma:[\ap,(\p,\q)] \to [\bp,(\p + \q)].$ The formula for the other morphisms is obvious.\\
For such $\gamma$ we have three cases:
$$\partial_{\n}\H(\gamma): \H(\p) \ox \H(\q) \to \H(\p+ \q); $$
$$\partial_{\n}\H(\gamma): \H(\p) \ox I \xrightarrow{\cong} \H(\p); $$
$$\partial_{\n}\H(\gamma): I \ox \H(\q) \xrightarrow{\cong} \H(\q); $$
\end{enumerate}

\end{schol}
\begin{defsans}
Let $\n>\0$ be an object of $\Phepiop$ and let $\H \in \Lax([\Phepiop]^{\leq (\n-\1)},\M)_{NA}$ be an $(\n-\1)$-truncated nonassociative lax diagram.\\

Define the lax-latching  object $L_\n^{\ox}\H$ of $\H$ at $\n$ as the colimit of the diagram
$$\partial_\n\H: \partial(Add \downarrow \n) \to \M.$$
\end{defsans}

A key ingredient is the following observation. 
\begin{lem}
The lax-latching object determines a left adjoint to the truncation from $\n$ to $\n-\1$:
$$\Lax([\Phepiop]^{\leq (\n)},\M)_{NA} \to \Lax([\Phepiop]^{\leq (\n-\1)},\M)_{NA}.$$
\end{lem}
\begin{proof}
Left to the reader.
\end{proof}
\begin{rmk}\label{rmk-latch-to-laxlatch}
Let $L_\n\H$ be the usual latching object of the underlying diagram $\H$ regarded as an object of $\Hom([\Phepiop]^{\leq (\n-\1)},\M)$ (see \cite{Hov-model}). 
It's not hard to see that there is a canonical map:
$$\delta_\n: L_n\H \to L_\n^{\ox}\H.$$
\end{rmk}
\subsection*{Left adjoint by induction}
We construct a left adjoint $\Gamma:  \Hom(\Phepiopl,\M) \to \Lax(\Phepiop,\M)_{NA}$ inductively as well as the unit $$\eta:\F \to \Ub \Gamma(\F).$$ To simplify the notation we will drop the functor $\Ub$ in $\eta$\\

Let $\F$ be an object in $\Hom(\Phepiopl,\M)$. As mentioned before we also regard $\F$ as an object of $\Hom(\Phepiop,\M)$ satisfying $\F(\0)=I$ (always).
\paragraph{Initialization} 
Let's set $\Gamma(\F)(\1)= \F(\1)$ and  $\Gamma(\F)(\0)= I= \F(\0)$.\\ 

Let $\eta_\1: \F^{\leq \1} \to \Gamma(\F)^{\leq \1} $ be the natural transformation given by the identity. 
\paragraph{From $\n-\1$ to $\n$} Assume that $\Gamma(\F)^{\leq \n-\1}$ is constructed and that we have a unit map 
$$\eta_{\n-\1}: \F^{ \leq \n-\1} \to \Gamma(\F)^{\leq \n-\1} $$

This map $\eta_{\n-\1}$ induces a canonical map between the classical latching objects
$$L_\n \F \to L_\n[\Gamma(\F)^{\leq \n-\1}].$$
If we compose it with the map $\delta_\n$ of Remark \ref{rmk-latch-to-laxlatch} we get a map
$$\xi_\n:  L_\n \F \to L_\n^{\ox}[\Gamma(\F)^{\leq \n-\1}].$$

Define $\Gamma(\F)(\n)$ as the pushout-object obtained by forming the pushout of the canonical map $L_\n\F \to \F(\n)$ along the attaching map 
$\xi_\n:L_\n \F \to L_\n^{\ox}[\Gamma(\F)^{\leq \n-\1}]$:
\[
\xy
(0,18)*+{L_\n\F}="W";
(0,0)*+{\F(\n) }="X";
(30,0)*+{\Gamma(\F)(\n)}="Y";
(30,18)*+{L_\n^{\ox}[\Gamma(\F)^{\leq \n-\1}]}="E";
{\ar@{->}^-{}"X";"Y"};
{\ar@{->}^-{}"W";"X"};
{\ar@{->}^-{\xi_\n}"W";"E"};
{\ar@{->}^-{}"E";"Y"};
\endxy
\]

The new map $\F(\n) \to \Gamma(\F)(\n)$ extends $\eta_{\n-\1}$ to $\eta_n$. It takes a little effort to show that the object $\Gamma(\F)$ that is constructed inductively this way, satisfies the universal property of a left adjoint. Since every construction is clearly functorial, we have a functor $\Gamma$.\\

Now it is classical to show that the functor $\Ub$ creates (hence preserves) coequalizer of $\Ub$-split pair. $\Ub$ also creates limits and filtered colimits, and it clearly reflects isomorphisms. By Beck monadicity we see that we have a monadic adjunction.\\

The induced monad is clearly finitary ($\Ub$ and $\Gamma$ preserve filtered colimits). Then following \cite[Remark 2.78]{Adamek-Rosicky-loc-pres} we get that $\Lax(\Phepiop,\M)_{NA}$ is also locally presentable. This ends the proof of the lemma.  $\qed$

\bibliographystyle{plain}
\bibliography{Bibliography_LP_COSEG}
\end{document}